\documentclass[14pt]{article}

\usepackage{amsmath}
\usepackage{amssymb}
\usepackage{amsthm}
\usepackage{bbm}
\usepackage{paralist}
\usepackage[all]{xy}

\newtheorem{Thm}{Theorem}

\newtheorem{theorem}{Theorem}[section]
\newtheorem{lemma}[theorem]{Lemma}
\newtheorem{corollary}[theorem]{Corollary}
\newtheorem{proposition}[theorem]{Proposition}

\newtheorem{example}[theorem]{Example}
\theoremstyle{definition}
\newtheorem{remark}[theorem]{Remark}
\newtheorem{definition}[theorem]{Definition}

\def\Z{\mathbb Z}
\def\C{\mathbb C}
\def\R{\mathbb R}

\def\0{\underline 0}

\begin{document}

\title {\bf  L\^e cycles and Milnor classes\footnote{Research partially supported by CAPES, CNPq and FAPESP,
Brazil, and by CONACYT and PAPIIT-UNAM, Mexico. \newline $\quad$
{\it Key-words:} Complex hypersurfaces, Milnor classes, L\^e
cycles, Whitney stratifications, constructible sheaves, polar
varieties, Schwartz-MacPherson classes, Fulton-Johnson classes.
\newline   {\it Mathematics Subject Classification.} Primary ;
32S50, 32S15, 14B05  Secondary; 14C17, 14J17, 55S35  }}

\vspace{1cm}
\author{R. Callejas-Bedregal, M. F. Z. Morgado and J. Seade }

\date{\bf }
\maketitle

\begin{abstract}

The purpose of this work is to establish a link between the theory
of Chern classes for singular varieties and the geometry  of the
varieties in question. Namely, we show that if $Z$ is a
hypersurface in a compact  complex manifold, defined by the
complex analytic space of zeroes of a reduced non-zero holomorphic section of a very ample line
bundle,  then its Milnor classes, regarded as elements in the Chow
group of $Z$, determine the  global L\^e cycles of $Z$; and
viceversa:  The L\^e cycles determine the Milnor classes. Morally
this implies, among other things, that the Milnor classes
determine  the topology of the local Milnor fibres  at each point
of $Z$, and  the geometry of the local Milnor fibres determines
the corresponding Milnor classes.

 \end{abstract}

 \section*{Introduction}

In this article,  a hypersurface  means a codimension one
analyticc subspace  of a complex manifold. We  consider compact
hypersurfaces defined by the complex
analytic space of zeroes of a reduced non-zero holomorphic section
of a very ample line bundle, and establish a deep relation between
their Milnor classes and their global L\^e cycles: We exhibit each
Milnor class as an explicit polynomial in the L\^e cycles and
viceversa.

To explain what this means, recall first that Chern classes of
complex ma-nifolds have played for decades a major role in complex
geometry and topology.  There are several extensions of this
important concept for  singular varieties, each having its own
interest and characteristics. Perhaps the most interesting of
these are  the Schwartz-MacPherson  and the Fulton-Johnson
classes. Milnor classes are global invariants of singular
hypersurfaces  in compact complex manifolds, elements in the Chow
group  of the singular variety, which measure the difference
between the Schwartz-MacPherson and the Fulton-Johnson classes.

The  literature about Milnor classes, and more generally about
Chern classes for singular varieties, is  large. And yet,  the
study of  the geometric and topological information that these
classes  encode is a branch of mathematics which still is in its
childhood. Contributing towards that goal is precisely the purpose
of this article.

L\^e cycles are analytic cycles  that describe, among other
things, the topology of the local Milnor fibres: We know from
\cite{Massey0, Massey} that there is a L\^e cycle in each
dimension, from $0$ to that of the singular set, and the
multiplicity of the L\^e cycles at each point says how many
handles of the corresponding dimension we must attach to a ball in
order to construct the local Milnor fibre (up to homeomorphism).

 L\^e cycles were  introduced by D. Massey in \cite{Massey0} for holomorphic map-germs
 $h:({\C}^{n+1},{\bf 0})\rightarrow (\C,0)$ by means of polar varieties.
 An interpretation of these cycles was given in \cite{GG} in terms of the Chern class of the tautological
 line bundle corresponding to the divisor determined by blowing up along the Jacobian ideal of $h,$
 which defines the critical set
 $\Sigma(h):=V(\frac{\partial h}{\partial z_0},\dots,\frac{\partial h}{\partial
 z_n})$ of $h.$
  Using \cite{GG}, the concept of L\^e cycles extends naturally to global singular hypersurfaces in compact complex manifolds.
  These are  elements in the Chow group, whose restriction to a neighborhood of each point are the  classes of the usual L\^e cycles.

Milnor classes were originally defined as elements in the singular
homology group of the hypersurface (or a complete intersection).
It was observed in the work of W. Fulton and  K. Johnson
(\cite{FJ}) and G. Kennedy (\cite{K}) that the Fulton-Johnson
classes and the Schwartz-MacPherson classes, and hence the Milnor
classes, can actually be considered as cycles in the corresponding
Chow group. Milnor classes have support in the singular set of the
variety, and there is a Milnor class in each complex dimension,
from 0 up to the dimension of the singular set.

Similarly, L\^e cycles also have support in the singular set of
the hypersurface and there is a L\^e cycle in each dimension, from
$0$ up to that of this singular set. This makes it natural to ask
what is the  relation among the global L\^e cycles and the Milnor
classes, and  that is the subject we explore in this article.

The concept of Milnor classes appeared first  in P. Aluffi's work
\cite{Aluffi0, Aluffi} under the name of $\mu$-class. Milnor
classes also appear implicitly in A. Parusi\'nski and P.  Pragacz'
article \cite {Par-Pr0}. The actual name of Milnor classes was
coined later
  by various authors at about the same time (see \cite{BLSS1, BLSS2, Yokura, Par-Pr});   the name comes from the fact that when the  singularities of
  $Z$ are all isolated, the Milnor class
  in dimension $0$, which is an integer, is the sum of the local Milnor numbers (by \cite{SS}), while  all other  Milnor classes are  $0$  (by  \cite{Suwa}).

  There are   interesting recent  articles about Milnor classes by various authors, as for instance P. Aluffi, J.-P. Brasselet, J. Sch\"{u}rmann, M. Tib\v ar, S. Yokura, T. Ohmoto, L. Maxim and others.
  There are also  important generalisations of this concept to different settings (see for instance \cite {Aluffi, CMSS, Yokura2}; also \cite{BSY}).
 Even so, Milnor  classes are still somehow mysterious objects, even ``esoteric".

 Milnor classes appear for instance in \cite{Alu-Mar} in relation with Feynman graphs in perturbative quantum field theory.
They appear also  in \cite{Beh} in relation with Donaldson-Thomas
type invariants for Calabi-Yau threefolds, via the classes
introduced by P. Aluffi in \cite {Aluffi1} for possibly nonreduced
spaces.  In the case of a hypersurface in a complex manifold, the
Aluffi classes are used in \cite{Aluffi1} to give an
interpretation of the Milnor classes.

  The aim of this note is two-folded: On the one hand we use Milnor classes to get global information about
   L\^e  cycles; and on the other hand we use L\^e  cycles  to get geometric and
   topological information about Milnor classes. Of course  this provides also a description of  the Aluffi classes in terms of L\^e cycles (Corollary \ref{c: Aluffi})
in the case of hypersurfaces in complex manifolds.

\medskip

  The main result in this article is the following theorem:

\begin{Thm} \label{principal}
Let $M$ be a compact complex manifold and $L$ a  very ample line
bundle on $M$. Consider the complex
analytic space $Z$ of zeroes  of a reduced non-zero holomorphic
section $s$ of $L$ and denote by $Z_{sing}$ the singular set of
$Z$.  Then the Milnor classes ${\cal M}_{k}(Z)$ and the L\^e
cycles $\Lambda_{k}(Z)$  of $Z$, $k= 0,\dots, d$ where
$d=\dim(Z_{sing})$, determine each other by the formulas:
$${\cal M}_{k}(Z)=\sum_{l=0}^{d-k}(-1)^{l+k}
\left( \begin{array}{c}l+k\\
k\\
\end{array}\right)
c_{1}(L|_{Z})^{l} \cap \Lambda_{l+k}(Z) ,$$ and
$${\Lambda}_{k}(Z)=\sum_{l=0}^{d-k}(-1)^{l+k}
\left( \begin{array}{c}l+k\\
k\\
\end{array}\right)
c_{1}(L|_{Z}))^{l} \cap {\cal M}_{l+k}(Z)\,, $$ for any
$k=0,\dots, d,$ where $c_1(L|_{Z})$ is the corresponding  Chern
class. These equalities are in the Chow group of  $Z$ and all the
cycles in them actually have support in
 $Z_{sing}.$
\end{Thm}

Since at each point in $Z$ the local   L\^e cycles determine the
topology, and hence the homology, of the local Milnor fibre $F$,
this theorem somehow indicates that the Milnor classes are
determined by the vanishing homology of $Z$, {\it i.e}, by the
kernel of the specialization morphism $H_*(Z_t) \to H_*(Z)$, where
$Z_t$ is the complex manifold defined by the  intersection  of
the zero section of the bundle $L$ with a section near the one
that determines $Z$, but which is everywhere transversal to the
zero section of  $L$. So  Theorem \ref {principal} gives a
positive answer to a question raised at the introduction in
\cite{BLSS2} (c.f. \cite {S1, S}).

\medskip
We also get  the corollary below, which extends and strengthens
\cite[Corollary 5.13]{BLSS2} in the hypersurface case:

     \medskip

    \noindent
    {\bf Corollary.}
{\it Assume $M, L $ and $Z$ are as above and equip $M$ with a
Whitney stratification $\{Z_\beta\}$ adapted to $Z$.  Let $d$ be
the dimension of the singular set $Z _{\rm sing}$.
 Then we have the following equalities of cycles in the Chow group of $Z$:
\begin{center}${\cal M}_{d}(Z) \; =\displaystyle\sum_{S_\beta \subset Z _{\rm sing}}
\;\mu^\perp(S_\beta)\; [\overline{S}_\beta] \;
=\displaystyle\sum_{S_\beta\subset Z _{\rm sing}}\;
\lambda^{d}_{S_\beta}\; [\overline{S}_\beta] \,= \, (-1)^d
\Lambda_{d}(Z) \,, $\end{center} where the sums run over the
strata of dimension $d$ which are contained in $Z _{\rm sing}$,
$\mu^\perp(S_\beta)$ is the transversal Milnor number of $S_\beta$
and $\lambda^{d}_{S_\beta}$ is the $d$-th L\^e number of
$S_\beta$. }
\medskip

The assumption  about the bundle $ L $ being very ample is used
to have
 a  projective embedding of $M$ such that $L$ is the pull back of the tautological bundle;
this leads to a description of the Schwartz-MacPherson classes in
terms of polar cycles. The ``very ampleness condition"  is also
needed to use R. Piene's theorem in \cite{Piene}, expressing the
Mather classes in polar terms. Yet,
  Theorem \ref{principal} describes intrinsically
the Milnor classes  in terms of  L\^e cycles and vice versa, and
it does not depend on the projective embedding of  $M$. This
suggests that the very ampleness condition on the line bundle $L$
may not be necessary.

The trail  for getting to Theorem 1 can be roughly described as
follows. The first step is  defining the global L\^e cycles. This
can be done in various ways. Here we do it by using the
interpretation of the local L\^e cycles given by T. Gaffney and R.
Gassler in \cite{GG}: We blow up the Jacobian ideal of the
hypersurface  $Z$ in $M$ and look at the Chern class of the
tautological bundle of the corresponding exceptional
divisor. This gives an explicit definition of
the global L\^e cycles.

Next we observe that the main theorem  of A. Parusinski and P.
Pragacz in \cite{Par-Pr} expresses the total
 Milnor class as a function of the Schwartz-MacPherson classes of the closure of the strata
 of a Whitney stratification (see equation (\ref{P})):

\begin{equation}{\cal M}(Z):=\sum_{S_\alpha \in {\cal
S}}\;\gamma_{S_\alpha}\left(c(L_{|_{Z}})^{-1} \cap
(i_{\bar{S_\alpha},Z})_{*}c^{SM}(\bar{S_\alpha})\right).\end{equation}
We refer to section 4 below for an explanation of the terms
involved in this formula.  On the other hand, J.-P. Brasselet in
\cite{BraC} conjectured that the Milnor classes can be expressed
in terms of polar varieties, which brings us closer to our goal of
comparing Milnor classes with L\^e cycles (which are defined via
polar varieties in \cite{Massey0, Massey}). { Following this path
we notice that  J. Sch\"urmann and  M. Tib\u ar introduced in
\cite{ST}, in the affine context, the MacPherson cycles associated
to any constructible function on a complex algebraic proper subset
$X\subset \mathbb{C}^{N}$. They showed that the corresponding
cycle class represents the (dual) Schwartz-MacPherson class  in
the Borel-Moore homology group, and also in  the  Chow group.}
We prove an analogous result in Section 3 in the projective case
(Theorem \ref{3}). In this construction a key role is played by
certain projective polar varieties.  This suffices for us, since
the assumption of considering a very ample line bundle $L$ over
$M$ implies that $M$ is a projective variety. We prove (see the
text for explanations):

\begin{Thm} \label{thm2} Let $X$ be an n-dimensional projective variety endowed with a Whitney stratification with connected strata $S_{\alpha}$.
 Let $\;i_{\overline{S}_{\alpha},X}:\overline{S}_{\alpha}\rightarrow
X$ be the na-tural inclusion. Consider $\varphi:X \rightarrow \C
P^{N}$ a closed immersion and ${\cal L}={\cal O}_{\C P^{N}}(1)$.
If $\beta:X\rightarrow \Z$ is a constructible function with
respect to this stratification, then the $k^{th}$
Schwartz-MacPherson class of $\beta$, $c_{k}^{SM}(\beta)$, is
given by:
$$c_{k}^{SM}(\beta)=\sum_{\alpha}\eta(S_{\alpha},\beta)\sum_{i=k}^{d_{\alpha}}(-1)^{d_{\alpha}-i}
\left( \begin{array}{c}i+1\\
k+1\\
\end{array} \right)(i_{\overline{S}_{\alpha},X})_{*}\left(c_{1}(\varphi^{*}{\cal L})^{i-k} \cap
[\mathbb{P}_{i}(\overline{S}_{\alpha})]\right),$$ where
$\eta(S_{\alpha},\beta)$ is the normal Morse index and $\mathbb{P}_{i}(\overline{S}_{\alpha})$ is the $i$-th (projective) polar variety of $\overline{S}_{\alpha}$.
\end{Thm}

{ A key point for proving Theorem  \ref{thm2} is
 Piene's   characterization in \cite[Th\'eor\`eme 3]{Piene}
of the Mather classes via polar varieties.

We remark that  there is another formula for the MacPherson
classes in terms of polar varieties given in \cite{Le-Teissier}.
Yet, the expression  we need for proving Theorem 1 is the one
given by Theorem
 \ref{thm2}, because this allows  comparison with the L\^e cycles.}

The final ingredient we need for proving Theorem 1 is  Massey's
concept of
  L\^e cycles for constructible sheaves via polar varieties  (see Remark \ref{Massey-polar}).  We call  these  Massey cycles. In Section 3 we extend this concept  to the projective setting and we prove a formula
  comparing  the   Massey  cycles   with the MacPherson cycles (Proposition \ref {p: Massey cycle}),
  analogous to Massey's formula in \cite[Theorem 7.5]{MasseyDuke}.

Theorem 1 is then proved by considering the   formula for the
Milnor classes by Parusinski-Pragacz, mentioned above, replacing
in it
 the Schwartz-MacPherson classes of $X$
 by the expression given
 in Theorem  \ref{thm2},    and then using Proposition \ref {p: Massey cycle} to express  the L\^e cycles in terms of  the MacPherson cycles in the projective setting. This is done in Section 5.

 We notice too  that the proof of  Theorem 1  leads to a description of the Milnor classes in terms of polar cycles, thus answering a question raised by J.-P. Brasselet  in \cite {BraC},  which was a motivation for this research  (see Remark
 \ref{Brasselet}).

The authors are grateful to David Massey, Marcelo Saia and J\"org
Sch\"urmann for very helpful conversations and comments. We
especially thank J. Sch\"urmann for pointing out that the
statement of Theorem $2$ given in the first version of this work
with $\beta=1_{X}$ could be improved for arbitrary $\beta$.
 The authors  are also
grateful to the Instituto de Ci\^encias Matem\'aticas e
Computa\c{c}\~ao of the University of S\~ao Paulo (USP)  at S\~ao
Carlos, Brasil, and to the Cuernavaca Unit of the Instituto de
Matem\'aticas of the National University of Mexico (UNAM), for
their support and hospitality while working on this article.

\section{Derived Categories}

We  assume some basic knowledge on  derived categories,
hypercohomology and  sheaves of vanishing cycles as described in
\cite{Dimca}.

If $X$ is a complex analytic space then ${\cal D}^{b}_{c}(X)$
denotes the derived category of bounded, constructible complexes
of sheaves of $\C$-vector spaces on $X$. We denote the objects of
${\cal D}^{b}_{c}(X)$ by something of the form $F^{\bullet}$. The
shifted complex $F^{\bullet}[l]$ is defined by
$(F^{\bullet}[l])^{k}=F^{l+k}$ and  its differential is
$d^{k}_{[l]}=(-1)^{l}d^{k+l}$. The constant sheaf $\C_{X}$ on $X$
induces an object $\C_{X}^{\bullet} \in {\cal D}^{b}_{c}(X)$ by
letting $\C_{X}^{0}=\C_{X}$ and $\C_{X}^{k}=0$ for $k\neq 0$.

If $h:X\rightarrow \C$ is an analytic map and $F^{\bullet}\in
{\cal D}^{b}_{c}(X)$, then we denote the sheaf of vanishing cycles
of $F^{\bullet}$  with respect to $h$ by $\phi_{h}F^{\bullet}$.

For $F^{\bullet}\in {\cal D}^{b}_{c}(X)$ and $p \in X$, we denote
by ${\cal H}^{*}(F^{\bullet})_{p}$ the stalk cohomology of
$F^{\bullet}$ at $p$, and by $\chi(F^{\bullet})_{p}$ its Euler
characteristic. That is,
$$\chi(F^{\bullet})_{p}=\sum_{k}(-1)^{k}{\rm dim}_{\C}{\cal
H}^{k}(F^{\bullet})_{p}.$$
 We also denote by $\chi(X,F^{\bullet})$ the Euler characteristic
of $X$ with coefficients in $F^{\bullet}$, {\it i.e.},
$$\chi(X,F^{\bullet})=\sum_{k}(-1)^{k}{\rm dim}_{\C}\mathbb{H}^{k}(X,F^{\bullet}),$$
where $\mathbb{H}^{*}(X,F^{\bullet})$ denotes the hypercohomology
groups of $X$ with coefficients in $F^{\bullet}$.

When $F^{\bullet}\in {\cal D}^{b}_{c}(X)$ is ${\cal
S}$-constructible, where ${\cal S}$ is a Whitney stratification of
$X$, we  denote it by $F^{\bullet}\in {\cal D}^{b}_{{\cal S}}(X)$.
We would like to point out the following result which appears in
\cite[Theorem 4.1.22]{Dimca}:
\begin{equation}\label{EulerCharact}\chi(X,F^{\bullet})=\sum_{S\in {\cal
S}}\chi(F^{\bullet}_{S})\chi(S),\end{equation} where
$\chi(F^{\bullet}_{S})=\chi(F^{\bullet})_{p}$ for an arbitrary
point $p \in S$.

For a complex analytic subspace $V$ of
$M$, we denote its conormal space by $T^{*}_{V}M$. That is,
$$T^{*}_{V}M:={\rm closure}\;\{ (x, \theta) \in T^{*}M\;|\; x \in
V_{{\rm reg}}\;{\rm and}\; \theta_{|_{T_{x} V_{{\rm reg}}}}\equiv
0\}\,,$$ where $T^{*}M$ is the cotangent bundle of $M$ and
$V_{{\rm reg}}$ is the regular part of $V$.

The following definition is standard in the literature:

\begin{definition}\label{constructible}
Let $ X $ be an analytic subspace of a complex manifold $ M $, $\{
S_{\alpha} \}$ a Whitney stratification of $M$ adapted to $ X $
and $x\in S_\alpha$ a point in $X$. Consider $g:(M,x)\rightarrow
(\C,0)$ a germ of holomorphic function such that $d_{x}g $ is a
{\it non-degenerate covector} at $x$ with respect to the fixed
stratification, that is, $d_{x}g \in T^{*}_{S_\alpha}M$ and
$d_{x}g \not\in T^{*}_{\overline{S^{'}\!\!}}\;M$, for all stratum
$S^{'} \neq S_\alpha$.  And let $N$ be  a germ of a closed complex
submanifold of $M$ which is transversal to $S_\alpha$, with $N
\cap S_\alpha=\{ x\}$. Define the {\it complex link }
$l_{S_\alpha}$ of $S_\alpha$ by:
 \begin{center}$l_{S_\alpha}:= X\cap N \cap
B_{\delta}(x)\cap \{g=w\}\quad{\rm for}\;0<|w|<\!\!< \delta<\!\!<
1.$\end{center} The {\it normal Morse datum} of $S_\alpha$ is
defined by:
$$NMD(S_\alpha):=(X\cap N \cap B_{\delta}(x),l_{S_\alpha}),$$ and the {\it normal Morse index} $\eta(S_\alpha,F^{\bullet})$  of the stratum is:
\begin{center}$\eta(S_\alpha,F^{\bullet}):=\chi(NMD(S_{\alpha}),F^{\bullet}),$\end{center}
where the right-hand-side means the Euler characteristic of the
relative hypercohomology. \end{definition}

By a result of M. Goresky and R. MacPherson in \cite[Theorem
2.3]{GM}  we get that the number $\eta(S_\alpha,F^{\bullet})$ does
not depends on the choices of $x\in S_\alpha,\; g$ and $N$.

Notice that by \cite[Remark 2.4.5(ii)]{Dimca}, it follows that
\begin{equation}\label{RelativeEuler}\eta(S_\alpha,F^{\bullet})=\chi(X\cap N \cap
B_{\delta}(x),F^{\bullet})-\chi(l_{S_\alpha},F^{\bullet})\,.\end{equation}

\begin{lemma}\label{indice} Let $F^{\bullet}\in {\cal D}^{b}_{{\cal
S}}(X)$ with ${\cal S}=\{S_{\alpha}\}$ a Whitney stratification of
$X$ and $g:X\rightarrow \C$ be a holomorphic function such that
$d_{p}g$  is a non-degenerate covector at $p \in S_\alpha$ with
respect to the fixed stratification. Set $d=\dim X$,
$d_{\alpha}=\dim S_{\alpha}$ and
$m_{\alpha}:=(-1)^{d-d_{\alpha}-1}\chi(\phi_{g_{|_{N}}}F_{|_{N}}^{\bullet})_{p}$,
where $\phi_{g_{|_{N}}}F_{|_{N}}^{\bullet}$ is the sheaf of
vanishing cycles of $F_{|_{N}}^{\bullet}$ with respect to
${g_{|_{N}}}$, $p \in S_{\alpha}$ and $N$ is a germ of a closed
complex submanifold which is transversal to $S_{\alpha}$ with $N
\cap S_{\alpha}=\{p\}$. Then
$$m_{\alpha}=(-1)^{d-d_{\alpha}}\eta(S_{\alpha},F^{\bullet}).$$
\end{lemma}

\begin{proof}
By \cite[Equation (4.1), p. 106]{Dimca} we have that
$${\cal H}^{i}(\phi_{g}F^{\bullet})_{p}\simeq \mathbb{H}^{i+1}(B_{\epsilon}(p) \cap
X, B_{\epsilon}(p) \cap X \cap g^{-1}(\varsigma),F^{\bullet}),
\,{\rm for}\;0<|\varsigma|<\!\!< \epsilon<\!\!< 1.$$ Hence
$$\chi(\phi_{g_{|_{N}}}F^{\bullet}_{|_{N}})_{p}=-\chi(B_{\epsilon}(p) \cap X \cap N, B_{\epsilon}(p) \cap X \cap N \cap
g^{-1}(\varsigma),F^{\bullet}),$$ and therefore
$m_{\alpha}=(-1)^{d-d_{\alpha}}\eta(S_{\alpha},F^{\bullet})$.
\end{proof}

\begin{remark}\label{ConstructibleFunction} Everything we have defined so far
for a constructible complex of sheaves is defined by J.
Sch\"{u}rmann and  M. Tib\u ar in \cite{ST} for constructible
functions, and the two
 constructions are somehow equivalent. In fact, given
$F^{\bullet}\in {\cal D}^{b}_{c}(X)$, we have naturally associated
the  constructible function on $X$ given by
$$\beta(p)
=\chi(F^{\bullet})_{p}.$$ Moreover, by Sch\"{u}rmann \cite{S}, the
converse also holds, {\it i.e.}, given any \linebreak
constructible function $\beta$ on $X$ there is $F^{\bullet}\in
{\cal D}^{b}_{c}(X)$ such that
$$\beta(p) =\chi(F^{\bullet})_{p}.$$
\end{remark}

The next result follows by Lemma \ref{indice} and \cite[Remark
10.5]{Massey}.

\begin{corollary}\label{local1} Let $h:X\rightarrow \C$ be an analytic map. If $P^{\bullet}=(\phi_{h}\mathbb{C}^{\bullet}_{X})[n-d]$ then
$m_{\alpha}=(-1)^{d-d_{\alpha}}\eta(S_{\alpha},P^{\bullet})=(-1)^{d-d_{\alpha}}\eta(S_{\alpha},w)$,
where $w$ is the constructible function defined by
$w(p)=\chi(P^{\bullet})_{p}$ $=\chi(F_{h,p})-1$ with $F_{h,p}$
being the Milnor fiber of $h$ at $p$.
\end{corollary}

\section{L\^e cycles}
\label{Le cycles}

\hspace{0.5cm} Let us recall first the definition of  L\^e cycles
and L\^e numbers of germs of complex analytic functions introduced
by D. Massey in \cite{Massey0} (see also \cite{Massey}).

We assume that the reader is familiar with the notion of gap
sheaves (see \cite{Si-Tr} and \cite[Definition 1.1]{Massey}). For
a coherent  sheaf of ideals $\alpha$ and an analytic subset $W$ in
an affine space $U$, we denote by $\alpha/W$ the corresponding gap
sheaf, which is a coherent sheaf of ideals in ${\cal O}_U,$ and by
$V(\alpha)/W$ the analytic space defined by the vanishing of
$\alpha/W.$ It is important to note that the analytic space
$V(\alpha)/W$ does not depend on the structure of $W$ as a scheme,
but only as an analytic set (see \cite[p. 10]{Massey}).

Let $U$ be an open subset of $\C^{n+1}$ containing the origin,
$h:(U,0)\rightarrow(\C,0)$ the germ of an analytic function,
$z=(z_{0},\cdots,z_{n})$ a linear choice of coordinates in
$\C^{n+1}$ and $\Sigma (h)=V\left(\frac{\partial h}{\partial
z_{0}},...,\frac{\partial h}{\partial z_{n}}\right)$ the critical
set of $h.$ To define the L\^e cycles we first need to define the
relative polar cycles, which are associated to the relative polar
varieties:   For each $k$ with $0 < k < n$, the polar variety
$\Gamma^{k}_{h,z}$ is
 the analytic space
$V\left(\frac{\partial h}{\partial z_{k}},...,\frac{\partial
h}{\partial z_{n}}\right)/ \; \Sigma (h),$ hence the analytic
structure of $\Gamma^{k}_{h,z}$ does not depend on the structure
of $\Sigma (h)$ as a scheme, but only as an analytic set. At
 the level of ideals, $\Gamma^{k}_{h,z}$ consists of those components of  $V\left(\frac{\partial h}{\partial
z_{k}},...,\frac{\partial h}{\partial z_{n}}\right)$ which are not
contained in the set $\Sigma (h)$. Massey
 denotes by $[\Gamma^{k}_{h,z}]$ the cycle associated with the
space $\Gamma^{k}_{h,z}$ (see \cite[p. 9]{Massey}).

 Then, for each $0 < k < n$ ,  Massey defines
the $k$-th L\^e cycle $\Lambda^{k}_{h,z}$ of $h$ with respect to
the coordinate system $z$ as the cycle:
\begin{center}$\Lambda^{k}_{h,z}:= \left[\;\Gamma_{h,z}^{k+1} \cap V\left(\frac{\partial
h}{\partial
z_{k}}\right)\;\right]-[\;\Gamma^{k}_{h,z}\;].$\end{center}

If a point $p = (p_{0},\cdots,p_{n}) \in U$ is an  isolated point
of the intersection of $\Lambda^{k}_{h,z}$ with the cycle of $V
(z_{0}-p_{0},\cdots,z_{k-1}-p_{k-1})$, then  the L\^e number
$\lambda^{k}_{h,z}(p)$ is the intersection number at $p$:
$$ \lambda^{k}_{h,z}(p):= (
\Lambda^{k}_{h,z} \cdot V(z_{0}-p_{0},..., z_{k-1}-p_{k-1})
\;)_{p}\;.$$

It is proved in \cite[Theorem 7.5]{MasseyDuke} (see also
\cite[Theorem 10.18]{Massey}) that for a generic choice of
coordinates, all  the L\^e numbers of $h$ at $p$ are defined and
they are independent of the choice of coordinates. Hence these are
called the ({\it generic}) {\it L\^e numbers} of $h$ at $p$ and
they are denoted simply by $\lambda_{h}^{k}(p)$.

\vspace{0,2cm}  Massey gives an alternative characterization of
the L\^e cycles  of a hypersurface singularity, which leads to a
generalization of the L\^e numbers that can be applied to any
constructible sheaf complexes. From this more general viewpoint,
the case of the L\^e numbers of a function $h$ is just the case
where the underlying constructible sheaf complexes is the sheaf of
vanishing cycles along $h$. Let us explain this.

Let $X$ be an analytic germ of a $d$-dimensional space which is
embedded in some affine space, $M := \mathbb{C}^{n+1}$, so that
the origin is a point of $X$. Consider  a bounded, constructible
sheaf  $F^{\bullet}$ on $X$ or $M$.

For a generic linear choice of coordinates $z = (z_{0},...,z_{n})$
for $\mathbb{C}^{n+1}$,  Massey in \cite[Proposition
0.1]{MasseyDuke} proves that there exist analytic cycles
$\Lambda_{i}(F^{\bullet})$ in X which are purely $i$-dimensional,
such that $\Lambda_{i}(F^{\bullet})$ and $V(z_{0}-x_{0},...,
z_{i-1}-x_{i-1})$ intersect properly at each point $x=(x_0,\cdots,x_n) \in X$ near
the origin, and such that for all x near the origin in $X$ we
have:
\begin{equation}\chi(F^{\bullet})_{x}=\sum_{i=0}^{d}(-1)^{d-i} \left( \Lambda_{i}(F^{\bullet}) \cdot V(z_{0}-x_{0},...,
z_{i-1}-x_{i-1})\right)_{x}  \;. \end{equation} Moreover, whenever
such analytic cycles $\Lambda_{i}(F^{\bullet})$ exist, they are
unique. He also  sets $\lambda_{F^{\bullet}}^{i}(x) = \left(
\Lambda_{i}(F^{\bullet}) \cdot V(z_{0}-x_{0},...,
z_{i-1}-x_{i-1})\right)_{x}$ and calls it the $i$-th
characteristic polar multiplicity of $F^{\bullet}$.

When $\beta(p)=\chi(F^{\bullet})_{p}$ we also denote
$\Lambda_{i}(F^{\bullet})$ by $\Lambda^{i}(\beta)$.

\begin{remark} \label{indice3} By \cite[Remark 10.5]{Massey0} it follows that
if we have $h:(U,0)\rightarrow (\mathbb{C},0)$ with $U$ an open
neighborhood of the origin in $\mathbb{C}^{n+1}$, $X=\Sigma(h)$
the critical set of $h$, $d=\dim_{0}X$ and we let
$$P^{\bullet}=(\phi_{h}\mathbb{C}^{\bullet}_{U})_{|_{\Sigma(h)}}[n-d] \,,$$
then for generic z, for all i and for all $x\in \Sigma(h)$ near
the origin, we have $\Lambda_{i}(P^{\bullet})=\Lambda_{h}^{i}$ and
$\lambda_{P^{\bullet}}^{i}(x)=\lambda_{h}^{i}(x)$.\end{remark}

\vspace{0,2cm}The work of T. Gaffney and R. Gassler \cite{GG}
describes the generic L\^e cycles in a coordinate-free way (see
also \cite[Corollary 7.9]{MasseyDuke}). For this, consider the
Jacobian blow up $\widetilde U$ of $h$ in $U$  with exceptional
divisor $D$.  One has a diagram:

\begin{center}$\left.\begin{array}{ccl}
D&\hookrightarrow&\widetilde U\\ \downarrow
~& ~&\downarrow b\\
\Sigma(h)&\hookrightarrow&U\\
\end{array}\right.$
\end{center}
Let $\mathcal L$ be the tautological line bundle of $\widetilde U$
and $c_{1}({\cal L}) \in H^2(\widetilde U) $  its Chern class.
Gaffney and  Gassler prove that one has:

$$\Lambda_{h}^{k}=b_{*}\left(c_{1}({\cal
L})^{n-k} \cap [D]\right).$$

This equality is as classes in the appropriate Chow group. For
simplicity, we denote by the same symbols the cycles and their
associated classes in the Chow group.

 Of course these
definitions and results extend naturally to map-germs defined on
open sets in complex manifolds. Moreover, the aforementioned
interpretation of the generic L\^e cycles, due to Gaffney and
Gassler, leads naturally to the following notions of L\^e cycles
associated to global hypersurfaces in complex manifolds.

Let $M$ be an $n+1$-dimensional compact complex analytic manifold,
and let $(L,M,\pi)$ be a holomorphic line bundle over $M$. Let $s$
be a holomorphic section of $L,$ generically transverse to the
zero-section of $L,$ and set $Z:=s^{-1}(0).$ That is, $Z$ is the
divisor of zeros of this section, which is a codimension one
subspace of $M$. Notice that $Z$ is locally defined by the set of
zeros of a holomorphic function, hence we call $Z$ a global
hypersurface of $M.$ This does not mean that $Z$ is defined by a
global equation. Let $Z_{\rm sing}$ and $\Sigma(s)$ be the
singular and critical space of $s$ respectively.  These are the
complex analytic subspaces of $M$
defined by the ideal $\left(h,\frac{\partial h}{\partial
z_{0}},...,\frac{\partial h}{\partial z_{n}}\right )$ and the
Jacobian ideal $J:=\left(\frac{\partial h}{\partial
z_{0}},...,\frac{\partial h}{\partial z_{n}}\right )$ of $h$
respectively, where $z_0,\dots,z_n$ are local parameters for $M$,
and $h$ is the local equation of the section of $L$ defining $Z$.
These structures are clearly independent of the choice of the
local parameters. Notice that $Z_{\rm sing}$ and $\Sigma(s)$ have
the same reduced structure \cite[p.3]{Aluffi0} and that $Z_{\rm
sing}$ is the set of points where the section $s$ fails to be
transversal to the zero-section of $L$. Consider the Jacobian blow
up $Bl_{J}M=B$ of $M$ along $\Sigma(s),$ let $D$ be the
exceptional divisor of $B$ and ${\cal L}$ the associated line
bundle on $B$, that we call the {\it tautological line bundle}  of
$B$. One has a diagram:

\begin{center}$\left.\begin{array}{ccl}
~&~&{\cal L}\\
~&~&\downarrow \\
D&\hookrightarrow&Bl_{J}M=B\\ \downarrow
~&~&\downarrow b\\
\Sigma(s)&\hookrightarrow&M\\
\end{array}\right.$
\end{center}

\begin{definition} For $0\leq k \leq n$, we define the $k$-th L\^e cycle of the section $s$ as the following class in the Chow group of $M$:
$$ \Lambda_{k}(Z)=b_{*}\left(c_{1}({\cal L})^{n-k} \cap [D]\right).$$\end{definition}

\vspace{0.2cm} Notice that $\Lambda_{k}(Z)$ is the $k$-dimensional
Segre class $s_k(\Sigma(s),M)$ of $\Sigma(s)$ in $M$, in the sense
of \cite[Section 4]{Ful}. It is observed in \cite[Lemma
1.1]{Aluffi0} that the total Segre class $s(\Sigma(s),M)$
coincides with the total Segre class $s(Z_{\rm sing},M).$  Hence
the definition of the L\^e cycles can be given in an analogous
manner by blowing up $M$ along $Z_{\rm sing}$ instead of blowing
up $M$ along $\Sigma(s).$ Thus $\Lambda_{k}(Z)$ is supported in
$Z_{\rm sing}$, so we will also look at $\Lambda_{k}(Z)$ as a
class in the Chow group of $Z_{\rm sing}$. In particular we have
that $\Lambda_{k}(Z)$ is zero for $k> {\rm dim}\;Z_{\rm sing}$.
Moreover, because $M$ is compact, these also represent classes in
the homology ring $H_{*}(M;\mathbb{Z})$.

The above discussion shows that restricted to a coordinate chart
we have:
\begin{equation}\label{globalLeCycle}{\Lambda_{k}(Z)}_{\mid_{U_{i}}}=\Lambda^{k}_{s_{i}},\end{equation}
as classes in the Chow group of $U_{i}$, where the right-hand side
is the class of the generic L\^e cycle defined by D. Massey and
where $\{(U_{i},s_{i})\}$ is a local description of $s$.

\section{Schwartz-MacPherson classes}

In this section we describe the Schwartz-MacPherson classes of a
complex analytic subspace of a
projective space via projective polar varieties.
%%%%%%%%%%%  CAMBIO AQUI: ME TRAJE LO QUE ESTABA EN LA SECCION SIGUIENTE
We recall first the classical construction defined   in
\cite{MacP}. Notice that this applies to
complex analytic subspaces in general, not only for hypersurfaces.

For any complex analytic subspace $X$
of a complex manifold $M$, consider  the Nash blow up $\tilde{X}
\stackrel{\nu}{\rightarrow}
 X$ of  $X$, its Nash bundle ${\tilde T} \stackrel{\pi}{\rightarrow}
\tilde{ X} $, and the Chern classes of $\tilde{T}$,
$c^{j}(\tilde{T}) \in H^{2j}(\tilde{ X})$, $j=1,\cdots,n$. The
Poincar\'e morphism $\beta_{\tilde{ X}}: H^{2p}(\tilde{
X})\stackrel{\cap[\tilde{ X}]}{\rightarrow} H_{2(n+1)-2p}(\tilde{
X})$, carries these into homology classes, which in turn can be
pushed forward into the homology of $ X$ via the homomorphism
$\nu_{*}$ induced by the projection. These are, by definition
\cite{MacP}, the Mather classes of $X$:
$$c^{Ma}_{k}( X):=v_{*}(c^{n-k}(\tilde{T})\cap [\tilde{ X}])\in H_{2k}( X),\;\;k=0,\cdots,n \,.$$
In fact one has  a homomorphism $c^{Ma}_{*}: Z(M)\rightarrow
H_{*}(M)$ defined in the obvious way, from the free abelian group
$Z(M)$ generated by the irreducible reduced complex analytic subspaces of $M$ to the
homology group (or Chow group, or Borel-Moore homology group). We
also define the dual Mather class $\check{c}^{Ma}_{*}(X)$ of $X$
by $\check{c}^{Ma}_{k}(X):=(-1)^{\dim X-k}c^{Ma}_{k}( X).$

The MacPherson classes are obtained from these by considering
appropriate ``weights" for each stratum, determined by the local
Euler obstruction $Eu_{{ X}}(x)$.  This is an integer associated
to each point $x \in X$. We refer to the literature for the
definition of the invariant $Eu_{{ X}}(x)$ (see for instance
\cite{MacP, BS} or \cite[Ch. 8]{BSS}).

 This invariant is constant in each Whitney stratum (see \cite{MacP, BS, Le-Teissier}), so  it defines a constructible function on $X$.
MacPherson shows in \cite{MacP} that there exists a unique set of
integers $b_{\alpha}$ for which the equation $\sum b_{\alpha}
Eu_{\overline{X}_{\alpha}}(x)=1$ is satisfied for all points $x
\in X$, where $\{ X_{\alpha}\}$ is a Whitney stratification of
$X$, the sum runs over all strata $ X_{\alpha}$ containing $x$ in
their closure and $Eu_{\overline{X}_{\alpha}}$ is the local Euler
obstruction of $\overline{ X}_{\alpha}$. Extending this by
linearity one gets a homomorphism $Eu: Z(M)\rightarrow F(M)$ from
the free abelian group $Z(M)$  to the group of constructible
functions on $X$.

We also define the dual Euler class $\check{E}u$ as the
transformation which associates to an irreducible analytic  subset
$X$ of $M$ the constructible function \linebreak $\check{E}u_{X}
:= (-1)^{\dim(X)} Eu_{X}$, determined by the local Euler
obstruction (see \cite{MacP} or \cite{GS}); this is  extended
linearly to cycles. Then $\check{E}u$ is an isomorphism of groups,
since
  $Eu_{X|_{Xreg}}$ is constant of value 1.

The MacPherson class of degree $k $ is defined by
$$c^{M}_{k}( X):=\sum
b_{\alpha}\;i_{*}(c^{Ma}_{k}(\overline{ X}_{\alpha})),$$ where $\{
X_{\alpha}\}$ is a Whitney stratification of $X$, $\overline{
X}_{\alpha}$ denotes the closure of the stratum, which is itself
analytic (and therefore it has its own Mather classes), and $i:
\overline{ X}_{\alpha}\hookrightarrow  X$ is the inclusion map.
Notice that these are homology classes by definition: $c^{M}_{k}(
X) \in H_{2k}(X; \mathbb X)$, $k= 0, 1, \cdots, n-1$. For
non-singular varieties, these  are the Poincar\'e duals of the
usual Chern classes.

 We notice too that by \cite{BS}, the MacPherson classes coincide, up to Alexander duality, with the
classes defined  by M.-H. Schwartz in \cite{Sch}. Thus, following
the modern literature (see for instance \cite{Par-Pr, BLSS2,
BSS}), we call these the  Schwartz-MacPherson classes of $X$ and
will be denoted from now on by $c^{SM}_{k}(X)$.

Let $L( M)$ be the free abelian group of all cycles generated by
the conormal spaces $T^{*}_{X} M$, where $X$ varies over all
complex analytic subspaces of $ M.$
Define the map  $cn:Z( M)\rightarrow L( M)$ by $cn(X):=T^{*}_{X}
M$. Clearly, this is an isomorphism. We also define the map $S:L(
M)\rightarrow H_*( M)$ by $S(X)=c^{*}(T^{*} M_{|_{X}})\cap
s_{*}(T^{*}_{X} M),$ where the Segre class $s_{*}$ is defined by
$$s_{*}(T^{*}_{X}M):=\widehat{\pi}^{'}_{*}\left(c^{*}({\cal
O}(-1))^{-1}\cap [P(T^{*}_{X}M)]\right)=\sum_{i\geq
0}{\pi}^{'}_{*}\left(c^{1}({\cal O}(1))^{i}\cap
[P(T^{*}_{X}M)]\right),$$ where ${\cal O}(1)$ denotes the
tautological line subbundle on the projectivisation \linebreak
$\widehat{\pi}^{'}:P(T^{*}M_{|_{X}})\rightarrow X$ with ${\cal
O}(-1)$ as its dual.

  By \cite[Lemme (1.2.1)]{Sabbah} (see also \cite[Lemma 1]{K}), one has a
description of the dual Mather class of $X$ in terms of the Segre
class of the conormal space $T^{*}_{X} M,$ as follows:
\begin{equation}\label{mather}\check{c}_{*}^{Ma}(X)=c^{*}(T^{*} M_{|_{X}})\cap
s_{*}(T^{*}_{X} M).\end{equation}

Finally we define a function $CC: F( M)\rightarrow L( M)$  by
$$CC(\alpha):=\sum_{S \in {\cal S}} (-1)^{\dim S}\eta(S,\alpha)\cdot  T^{*}_{\overline{S}} M,$$
where $\alpha $ is a constructible function on $ M$ with respect
to the Whitney stra-tification $\cal S$ of $ M$ and
$\eta(S,\alpha)$ is the normal Morse index $\eta(S,F^{\bullet}),$
(see Definition \ref{constructible}) where $F^{\bullet}$ is the
constructible complex of sheaves such that $\alpha(p)
=\chi(F^{\bullet})_{p}$ (see Remark \ref{ConstructibleFunction}).
 One can prove that this map $CC$ is an isomorphism (see for example the
 discussion on \cite{ST} and the reference therein).

From the above discussions we have the following commutative
diagram:

\begin{equation}\label{diagrama}
\xymatrix { F(M)\ar[d]^{\mbox{id}} &Z(M) \ar[l]_{\check{E}u}\ar[r]^{\check{c}_{*}^{Ma}} \ar[d]^{cn} & H_{*}(M)\ar[d]^{\mbox{id}}\\
F(M) \ar[r]^{CC} &  L(M)  \ar[r]^{S}  & H_{*}(M) }\end{equation}

The commutativity of the left square of this diagram amounts to
saying:
$$\beta=\sum_{\alpha}(-1)^{d_{\alpha}}\eta(S_{\alpha},\beta) \cdot
\check{E}u_{S_{\alpha}},$$ for any function $\beta:X\rightarrow
\Z$ which is constructible for the given Whitney stratification.
Using the commutativity of the other part of the diagram, we have
that
\begin{equation}\label{SMclass}\check{c}_{k}^{SM}(\beta)=\sum_{\alpha}(-1)^{
d_{\alpha}}\eta(S_{\alpha},\beta)(i_{\overline{S}_{\alpha},X})_{*}\check{c}_{k}^{Ma}(\overline{S}_{\alpha}).\end{equation}

\vspace{0.2cm} In the affine context,  Sch\"{u}rmann and  Tib\u ar
in \cite{ST} describe the Schwartz-MacPherson classes of a complex
algebraic proper subset $X\subset \mathbb{C}^{N}$ using algebraic
cycles, which were called MacPherson cycles. In this construction
a key role is played by the affine polar varieties, which we now
describe.

If $X$ is of pure dimension $n<N$, the {\it $k$-th global polar
variety} of $X,$ \linebreak $0\leq k\leq n,$ is the following
algebraic set (see \cite{Le-Teissier})
$$P_{k}(X)=\overline{Crit(x_{1},\cdots,x_{k+1})_{|_{X_{reg}}}} \;,$$
with $Crit(x_{1},\cdots,x_{k+1})_{|_{X_{reg}}}$ the usual critical
locus of points $x \in X_{reg}$ where the differentials of these
functions restricted to $X_{reg}$ are linearly dependent. For
general coordinates $x_{i}$, the polar variety $P_{k}(X)$ has pure
dimension k or it is empty, for all $0< k< n$. We have
$P_{n}(X):=X$ and we set $P_{k}(X):=\emptyset$ for $k > n$.

We fix an algebraic Whitney stratification ${\cal S}$ with
connected strata. In this context $X$ need not be pure dimensional
and we only assume $n =\dim X < N$. Let $\alpha: X \rightarrow \Z$
be an ${\cal S}$-constructible function, meaning that the
restriction $\alpha_{|_{S}}$ is constant for all strata $S\in
{\cal S}$.

Sch\"{u}rmann and  Tib\u ar define the $k$-th MacPherson cycle of
$\alpha$ ($0\leq k\leq n$) by:
\begin{equation}MP_{k}(\alpha):=\displaystyle\sum_{S\in {\cal S}}(-1)^{{\rm
dim}\;S}\eta(S,\alpha)(\psi_{S})_{*}[P_{k}(\overline{S})],\end{equation}
where $\psi_{S}$ is the inclusion of $\overline{S}$ in $X$ and
$P_{k}(\overline{S})$ is the $k$-th global polar variety of the
algebraic closure $\overline{S}\subset\mathbb{C}^{N}$ of the
stratum $S$.

\begin{remark}\label{Massey-polar}
Let $\{S_{\alpha}\}$ be any Whitney stratification of $X\subset
\C^{N}$ (with connected strata) with respect to which
$F^{\bullet}$ is constructible. Massey
 obtains an important characterization of the
cycles $\Lambda_{i}(F^{\bullet})$ using the polar varieties
\linebreak (\cite[Corollary 10.15]{Massey} or \cite{MasseyDuke}).
We call the $k^{th}$ {\it Massey cycle of the constructible sheaf}
$F^{\bullet}$
 the cycle
\begin{equation}\label{indice1}\Lambda_{k}(F^{\bullet})=\sum_{\alpha} m_{\alpha}
(\psi_{\alpha})_{*}[P_{k}(\overline{S}_{\alpha})],\end{equation}
 where $\psi_{\alpha}:\overline{S}_{\alpha}\hookrightarrow X$ is the
inclusion,
$m_{\alpha}=(-1)^{d-d_{\alpha}-1}\chi(\phi_{g_{|_{N}}}F_{|_{N}}^{\bullet})_{p}$,
with $g$ being a non-degenerate covector at $p \in S_\alpha$ with
respect to the fixed stratification,
 $N$ is a germ of a closed complex submanifold
which is transversal to $S_{\alpha}$ with $N \cap
S_{\alpha}=\{p\}$, and $P_{k}(\overline{S}_{\alpha})$ denotes the
(affine) polar variety of dimension k of $\overline{S}_{\alpha}$.
Hence, by Lemma \ref{indice} and Remark
\ref{ConstructibleFunction}, we have that
\begin{equation}\Lambda_{k}(F^{\bullet})=(-1)^{d}MP_{k}(\beta),\end{equation}
where $\beta(p)=\chi(F^{\bullet})_{p}$ is a constructible function
on $X$ and $d=\dim X$.

%In particular, if $h:U\rightarrow \C$ is holomorphic function, where $U \subset \C^{n+1}$ is a open set, and $F^{\bullet}=P^{\bullet}=(\phi_{h}\mathbb{C}^{\bullet}_{U})_{|_{\Sigma(h)}}[n-d]$,
% by the Corollary \ref{local1}, we have that
%\begin{equation}\label{local} \Lambda_{h}^{k}=\Lambda_{k}(P^{\bullet})=(-1)^{d}MP_{k}(w).\end{equation}
%where $w(p)=\chi(P^{\bullet})_{p}=\chi(F_{h,x})-1$ is a constructible function on $\Sigma (h)$ and $d=\dim \Sigma (h)$.

\end{remark}

The most important result of \cite{ST} is that the cycle
 $MP_{k}(\alpha)$ represents the $k$-th dual
Schwartz-MacPherson class $\check{c}_{k}^{SM}(\alpha)$ in the Chow
group $A_{k}(X)$, where \linebreak
$\check{c}_{k}^{SM}(\alpha)=(-1)^{k}c_{k}^{SM}(\alpha)$. That is,
$$c_{k}^{SM}(\alpha)=(-1)^{k}MP_{k}(\alpha)=(-1)^{k}\displaystyle\sum_{S\in
{\cal S}}(-1)^{{\rm
dim}\;S}\eta(S,\alpha)(\psi_{S})_{*}[P_{k}(\overline{S})].$$ Hence
 Sch\"{u}rmann and  Tib\u ar describe in this way the Schwartz-MacPherson classes via
affine polar varieties.

These definitions and results motivate the following construction
in the projective context.

\begin{definition} Let $X$ be a complex analytic space in $\C P^{N}$ of pure dimension $d$. The $k$-th polar variety of $X$ is defined by
\begin{center}$\mathbb{P}_{k}(X)=\overline{\{ x \in X_{reg} \; | \; {\rm dim} \left(T_{x}X_{reg}\cap L_{k+2}\right) \geq d-k-1 \}}$,\end{center}
where $L_{k+2}$ is a fixed plane of the codimension ${k+2}$ in $\C
P^{N}$  and $T_xX_{reg}$ is the projective tangent space of $X$ at
a regular point $x$.
\end{definition}

We observe that this definition is an extension of the local
definition given in  \cite{Le-Teissier}. Notice also that for
$L_{k+2}$ sufficiently general, the dimension of
$\mathbb{P}_{k}(X)$ in $X$ is equal to $k$. {Thus we are indexing
the polar varieties by their dimension and not by their
codimension, as it is usually done.} Observe also that the class
$[\mathbb{P}_{k}(X)]$ of $\mathbb{P}_{k}(X)$ modulo rational
equivalence in the Chow group $A_{k}(X)$ does not depend on
$L_{k+2}$ provided this is sufficiently general.

\begin{definition}
The class  $[\mathbb{P}_{k}(X)]$ is the {\it $k$-th polar class
of} $X$.
\end{definition}

For any given $F^{\bullet}\in {\cal D}^{b}_{{\cal S}}(X)$, where
${\cal S}=\{S_{\alpha}\}$ is a Whitney stratification of $X$,
define the MacPherson cycle
\begin{equation}\label{GlobalCycle}MP_{k}(F^{\bullet}):=\displaystyle\sum_{\alpha}(-1)^{{\rm
dim}\;S_{\alpha}}\eta(S_{\alpha},F^{\bullet})(\psi_{\alpha})_{*}[\mathbb{P}_{k}(\overline{S}_{\alpha})],\end{equation}
where $\psi_{\alpha}:\overline{S}_{\alpha}\hookrightarrow X$ is
the inclusion. Sometimes, we also write this cycles by
$MP^{\mathbb{P}}_{k}(F^{\bullet})$ and
$\eta(S_{\alpha},F^{\bullet})$ by
$\eta^{\mathbb{P}}(S_{\alpha},F^{\bullet})$, in order to emphasize
the projective nature of this cycle and of this number. If $\beta$
is the constructible function associated to $F^{\bullet}$ as in
Remark $\ref{ConstructibleFunction}$ we also denote this cycle
$MP^{\mathbb{P}}_{k}(F^{\bullet})$ by
$MP^{\mathbb{P}}_{k}(\beta)$.

Analogously, motivated by equation (\ref{indice1}), we define:

\begin{definition}\label{indice2} The {\it projective Massey cycle of }
$F^{\bullet}$ is:
$$\Lambda^{\mathbb{P}}_{k}(F^{\bullet})=\sum_{\alpha} m_{\alpha}
(\psi_{\alpha})_{*}[\mathbb{P}_{k}(\overline{S}_{\alpha})], $$
where $\psi_{\alpha}:\overline{S}_{\alpha}\hookrightarrow X$ is
the inclusion and
$m_{\alpha}=(-1)^{d-d_{\alpha}-1}\chi(\phi_{g_{|_{N}}}F_{|_{N}}^{\bullet})_{p}$\,;
 $g$ is a non-degenerate covector at $p \in S_\alpha$ with
respect to the fixed stratification ${\cal S}=\{S_\alpha\}$,
 $N$ is a germ of a closed complex submanifold
which is transversal to $S_{\alpha}$ with $N \cap
S_{\alpha}=\{p\}$, and $\mathbb{P}_{k}(\overline{S}_{\alpha})$
denotes the (projective) polar variety of dimension $k$ of
$\overline{S}_{\alpha}$. When $\beta(p) =\chi(F^{\bullet})_{p}$
for $F^{\bullet}\in {\cal D}^{b}_{{\cal S}}(X)$ we denote
$\Lambda^{\mathbb{P}}_{k}(F^{\bullet})$ also by
$\Lambda^{\mathbb{P}}_{k}(\beta).$
\end{definition}

\begin{remark} We notice  that the projective polar variety coincides with
the compactification of the affine polar variety. In fact there is
a Zariski open dense $U$ of hyperplanes of $\C P^{N}$ such that
for all $H \in U$ we have that
$$ \mathbb{P}_{k}(X) \cap \left( \mathbb{P}^{N}\setminus H
\right)=P_{k}(X \cap (\mathbb{P}^{N}\setminus H))\,.$$
\end{remark}

Consequently,
$$i^{*}(MP_{k}^{\mathbb{P}}(F^{\bullet}))=MP_{k}(F^{\bullet}_{|_{X \cap (\mathbb{P}^{N}\setminus H)}})\quad {\rm and}\quad
i^{*}(\Lambda_{k}^{\mathbb{P}}(F^{\bullet}))=\Lambda_{k}(F^{\bullet}_{|_{X
\cap (\mathbb{P}^{N}\setminus H)}}),$$ where
$i:\mathbb{P}^{N}\setminus H\hookrightarrow \mathbb{P}^{N}$ for
all $H \in U^{'}\subset U$ and $U^{'}$ is such that if $H \in
U^{'}$, $F^{\bullet}_{|_{X \cap (\mathbb{P}^{N}\setminus H)}}$ is
constructible  with respect  to the induced Whitney stratification
of $X\setminus H$ with strata
 $${\cal S}_{|_{X \cap (\mathbb{P}^{N}\setminus H)}}:=\{S\setminus H \;|\: S \in
{\cal S}\}.$$ Hence, for all $p \in (\mathbb{P}^{N}\setminus H)
\cap X$ we have that
 $\eta^{\mathbb{P}}(S,F^{\bullet})=\eta(S,F^{\bullet}_{|_{X \cap
(\mathbb{P}^{N}\setminus H)}}),$ for all $S \in {\cal S}_{|_{X
\cap (\mathbb{P}^{N}\setminus H)}}$.

Since two cycles on $X$ agree if and only if their pull-back to
every affine open covering of $X$ agree,  by equation
(\ref{indice1}), Remark \ref{indice3},  and the above discussion
we have:
\begin{proposition} \label{p: Massey cycle} The
Massey cycle and the MacPherson cycle agree up to sign. That is:
\begin{equation}\label{indice4}
\Lambda_{k}^{\mathbb{P}}(F^{\bullet})=(-1)^{d}MP_{k}^{\mathbb{P}}(F^{\bullet}),
\end{equation}  where $d=\dim X$.  And if $X$ is a hypersurface $Z$, then we have:
\begin{equation}\label{5}(i_{sing})_{*}\left(\Lambda_{k}(Z)\right)=\Lambda^{\mathbb{P}}_{k}(w),\end{equation} where $w$ is the constructible function on $Z$ defined as in Corollary \ref{local1} by $w(p)=\chi(F_{s,p})-1$ and
$i_{sing}:Z_{sing}\rightarrow Z$ is the natural inclusion.
\end{proposition}

Notice that this result actually gives a cycle structure to the
L\^e cycle $\Lambda_{k}(Z)$, which is   defined only  as a class
in the Chow group of $Z_{\rm sing}$. \vspace{0.2cm} The next
result, stated as Theorem 2 in the introduction, is the projective
analogous of the formula in \cite{ST} describing
 the Schwartz-MacPherson class in terms of polar cycles.

\begin{theorem} \label{3}Let $X$ be an n-dimensional projective variety endowed with a Whitney stratification with connected strata $S_{\alpha}$.
 Let $\;i_{\overline{S}_{\alpha},X}:\overline{S}_{\alpha}\rightarrow
X$ be the natural inclusion. Consider $\varphi:X \rightarrow \C
P^{N}$ a closed immersion and ${\cal L}={\cal O}_{\C P^{N}}(1)$.
If $\beta:X\rightarrow \Z$ is a constructible function with
respect to this stratification, then
$$c_{k}^{SM}(\beta)=\sum_{\alpha}\eta(S_{\alpha},\beta)\sum_{i=k}^{d_{\alpha}}(-1)^{d_{\alpha}-i}
\left( \begin{array}{c}i+1\\
k+1\\
\end{array} \right)(i_{\overline{S}_{\alpha},X})_{*}\left(c_{1}(\varphi^{*}{\cal L})^{i-k} \cap
[\mathbb{P}_{i}(\overline{S}_{\alpha})]\right)$$
$$=\sum_{i\geq k}(-1)^{i}
\left( \begin{array}{c}i+1\\
k+1\\
\end{array} \right)\left(c_{1}(\varphi^{*}{\cal L})^{i-k} \cap
MP^{\mathbb{P}}_{i}(\beta)\right),$$
where $[\mathbb{P}_{i}(\overline{S}_{\alpha})]$ is the $i$-th polar class of  $\overline{S}_{\alpha}$  and $MP^{\mathbb{P}}_{i}(\beta)$ is the $i$-th MacPherson cycle of $\beta$.
\end{theorem}

\begin{proof} For any purely dimensional projective variety $V$ of dimension
$d$ we have, by R. Piene's work \cite{Piene}, the following
characterization of the Mather classes via polar varieties:

\begin{equation}\label{1}c_{k}^{Ma}(V)=\displaystyle\sum_{i=k}^{n} (-1)^{n-i}
\left( \begin{array}{c}i+1\\
k+1\\
\end{array} \right)\left(c_{1}(\varphi^{*}{\cal L})^{i-k} \cap
[\mathbb{P}_{i}(V)]\right).\end{equation} By equation
(\ref{SMclass}) we have that
\begin{equation}\label{2}\check{c}_{k}^{SM}(\beta)=\sum_{\alpha}(-1)^{
d_{\alpha}}\eta(S_{\alpha},\beta)(i_{\overline{S}_{\alpha},X})_{*}\check{c}_{k}^{Ma}(\overline{S}_{\alpha}).\end{equation}
Therefore by equations (\ref{1}), (\ref{2}) and the relationship
between the MacPherson class and its dual class, {\small
$$c_{k}^{SM}(\beta)=\sum_{\alpha}\eta(S_{\alpha},\beta)\sum_{i=k}^{d_{\alpha}}(-1)^{d_{\alpha}-i}
\left( \begin{array}{c}i+1\\
k+1\\
\end{array} \right)(i_{\overline{S}_{\alpha},X})_{*}\left(c_{1}(\varphi^{*}{\cal L})^{i-k} \cap
[\mathbb{P}_{i}(\overline{S}_{\alpha})]\right)$$

$$=\sum_{i\geq k}(-1)^{i}
\left( \begin{array}{c}i+1\\
k+1\\
\end{array} \right)\left(c_{1}(\varphi^{*}{\cal L})^{i-k} \cap
MP^{\mathbb{P}}_{i}(\beta)\right).$$}
 \end{proof}

\section{Milnor classes}

 Let $M$ be as before, a  compact complex manifold of
dimension $n+1$. We  consider a singular hypersurface $Z$ defined
by a holomorphic section $s$  of some line bundle $L$ over $M$.
The Milnor classes of $Z$  are defined as the difference between
the Schwartz-MacPherson classes and the Fulton-Johnson classes of
$Z$, as we explain below.

The  Schwartz-MacPherson classes were described in the previous
section, and they provide an extension for singular varieties of
the classical Chern classes of complex manifolds. There is another
natural
 generalisation of the  Chern classes due to W. Fulton and K. Johnson in \cite{FJ}.
In the case we consider here, where $Z$ is a hypersurface,  this
uses
 the virtual tangent bundle $\tau(Z)$ of $Z$, which plays the role of
the tangent bundle. This virtual bundle is by definition:
$$\tau(Z):= T M_{|_{Z}}-L_{|_{Z}}\,,$$
where $T M$ denotes the tangent bundle of $ M$, $L$ is the bundle
of $Z$ and the difference is in the  KU-theory of $Z$. Notice that
restricted to the regular part of $Z$, the bundle $L$ is
isomorphic to the normal bundle of $Z$ in $M$.
 The Chern classes of the virtual  bundle $\tau(Z)$ live in the cohomology of $Z$, and the Poincar\'e morphism carries them into homology classes:
$$c^{FJ}(Z):=c(T M_{|_{Z}}-L_{|_{Z}})\cap [Z],$$
where  the total  Chern class of $(T M_{|_{Z}}-L_{|_{Z}})$ is
defined by $c^{}(T M_{|_{Z}}-L_{|_{Z}})=i^{*}(c^{}(T M)\cdot
c^{}(L)^{-1})$, with $i:Z\hookrightarrow  M $  being the
inclusion.

\begin{definition} (\cite{Aluffi, BLSS1, BLSS2, Yokura, Par-Pr}) The {\it total Milnor class} of $Z$ is defined by:
$${\cal M}(Z):=(-1)^{n-1}(c^{FJ}(Z)-c^{SM}(Z))\,.$$
We denote by ${\cal M}_k(Z)$ the component of this total class in
$H_{2k}(Z)$, $k= 0,\cdots, 2n$,  and call it {\it the Milnor class
of $Z$ of degree} $k$.
\end{definition}

Milnor classes were originally defined as elements in the singular
homology group of the hypersurface (or a complete intersection).
It was observed in the work of W. Fulton and  K. Johnson
(\cite{FJ}) and G. Kennedy (\cite{K}) that the Fulton-Johnson
classes and the Schwartz-MacPherson classes, and hence the Milnor
classes, can actually be considered as cycle classes in the
corresponding Chow group.

In the sequel we will use the following  important
characterization of Milnor classes obtained by A. Parusinski and
P. Pragacz in \cite[Theorem 0.2]{Par-Pr}:

\begin{equation}\label{P}{\cal M}(Z):=\sum_{S \in {\cal
S}}\;\gamma_{S}\left(c(L_{|_{Z}})^{-1} \cap
(i_{\overline{S},Z})_{*}c^{SM}(\overline{S})\right),\end{equation}
where ${\cal S}=\{ S\}$ is an analytic  Whitney stratification of
$Z$ with connected strata, such that  $Z _{\rm sing}$ is union of
strata, $i_{\overline{S},Z}:\overline{S}\hookrightarrow Z$ is the
inclusion and $\gamma_{S}$ is the function defined on each stratum
$S$ as follows: For each $x \in S \subset Z$, let $F_x$ be a {\it
local Milnor fibre} (recall $Z$ is a hypersurface in the complex
manifold $M$), and  let  $\chi(F_{x})$ be its Euler
characteristic. We set:
$$\mu(x;Z):= (-1)^{n}\;(\chi(F_{x})-1) \,,$$ and
call it the {\it local Milnor number} of $Z$ at $x$. This number
is constant on each Whitney stratum, by the topological triviality
of Whitney stratifications, so we denote it by $\mu_{S}$. Then
$\gamma_{S}$ is defined inductively by:
$$\gamma_{S}=\mu_{S} - \sum_{S' \neq S,\;\overline{S'} \supset S}
\gamma_{S{'}}.$$

For instance, suppose that the singular set of $Z$ is a finite
number of points $x_{1},\cdots,x_{r}$, and let $\mu_{x_{i}}$
denote the corresponding Milnor numbers. Then by \cite{SS} we have
${\cal M}_0(Z)=\sum_{i=1}^{r}\mu_{x_{i}}[x_{i}] \in H_{0}(Z)$, and
 all other Milnor classes are $0$ by \cite{Suwa}. More generally
one has the following result \cite[Theorem 5.2 and Corollary
5.13]{BLSS1, BLSS2},  that we  prove here for completeness, using
the description of  Milnor classes  by A. Parusinski and P.
Pragacz (see equation (\ref{P})).

\begin{proposition}\label{t. localization}
The Milnor classes have support in  the singular set $Z _{\rm
sing}$. That is, $${\cal M}(Z)=\sum_{S \subset Z _{\rm
sing}}\;\gamma_{S}\left(c(L_{|_{Z}})^{-1} \cap
(i_{\bar{S},Z})_{*}c^{SM}(\bar{S})\right)\,.$$ Hence these classes
are all zero in dimensions above the dimension of $Z _{\rm
sing}$.\end{proposition}

\begin{proof} Let ${\cal S}=\{ S\}$ be an analytic  Whitney
stratification of $Z$ with connected strata, such that   $Z _{\rm
sing}$ is union of strata. Notice that the closure $\overline S$
of each stratum  is again analytic, so it has its own
Schwartz-MacPherson classes $c^{SM}(\overline{S})$. By equation
(\ref{P}) we have that
$${\cal M}(Z):=\sum_{S \in {\cal
S}}\;\gamma_{S}\left(c(L_{|_{Z}})^{-1} \cap
(i_{\overline{S},Z})_{*}c^{SM}(\overline{S})\right).$$ If $ x $ is
a regular point of $Z$, then $\chi(F_{x})=1$ and $\mu_{S}=0$.
Hence
$${\cal M}(Z)=\sum_{S \subset Z _{\rm sing}}\;\gamma_{S}\left(c(L_{|_{Z}})^{-1} \cap (i_{\overline{S},Z})_{*}c^{SM}(\overline{S})\right)\,.$$
Since for strata in the singular set of $Z$ one has ${\rm dim}\;
\overline{S}\leq {\rm dim}(Z _{\rm sing})$, then $c^{SM}_{k}(S)=0$
if $k>{\rm dim}(Z _{\rm sing})$. Therefore ${\cal M}_{k}(Z)=0$ if
$k>{\rm dim}(Z _{\rm sing})$.
\end{proof}

\vspace{0.3cm} Thus, if the dimension of $Z _{\rm sing}$ is $d$,
then we have Milnor classes ${\cal M}_{k}(Z) \in H_{2k}(Z)$, $k =
0, \cdots, d$,  which are non-zero generally speaking, and each
has support in $Z_{\rm sing}$.

\section{Milnor classes via L\^e cycles}

This section is devoted to proving Theorem \ref{principal}. We
first prove:

\begin{lemma}[{\bf Main Lemma}]\label{l: main}
 Let  $M^{n+1}$ be a compact complex manifold, consider  a very ample line bundle $L$ on $M$
 and let  $Z$ be the zero set of a holomorphic section
$s$ of  $L$. Then the corresponding Milnor classes and L\^e cycles
satisfy:
$${\cal M}_{k}(Z)=  \sum_{j\geq
0}\sum_{i\geq k+j}(-1)^{i+j}
\left( \begin{array}{c}i+1\\
k+j+1\\
\end{array} \right)
c_{1}(L|_Z)^{i-k} \cap \Lambda_{i}(Z) .$$ This equality is in the
Chow group of $Z$ and, furthermore, all cycles in this formula
actually have support in the singular set of $Z$.
\end{lemma}

\begin{proof}
Let ${\cal S}=\{S_{\alpha} \}$ be a Whitney stratification of $ Z
$ such that each stratum is connected and $Z_{{\rm sing}}$ is a
union of strata.

 Since $L$ is very ample, there is a closed immersion
$f:M\hookrightarrow \C P^{N}$ such that $L=f^{*}({\cal O}_{\C
P^{N}}(1))$, where ${\cal O}_{\C P^{N}}(1)$ is the tautological
bundle of $\C P^{N}$. Set ${\cal L}={\cal O}_{\C P^{N}}(1)$.

Consider the following inclusions and the closed immersion $f$:
$$\overline{S}_{\beta}\stackrel{i_{\overline{S}_{\beta},\overline{S}_{\alpha}}}{\hookrightarrow}
\overline{S}_{\alpha}\stackrel{\psi_{\alpha}}{\hookrightarrow}
Z\stackrel{g}{\hookrightarrow}M\stackrel{f}{\hookrightarrow}\C
P^{N} \,.$$

Since the mappings are closed (and therefore proper) and $\C
P^{n}$ is a compact complex manifold, we have the following
characterization of the Milnor classes of $Z$ given by A.
Parusinski and P. Pragacz (see equation (\ref{P})):
\begin{equation}\label{PP}{\cal
M}_{k}(Z)=\sum_{\alpha}\gamma(S_{\alpha})
\left(\sum_{j=0}^{d_{\alpha}-k}(-1)^{j}c_{1}(g^{*}f^{*}{\cal
L})^{j} \cap
(\psi_{\alpha})_{*}c_{k+j}^{SM}(\overline{S}_{\alpha})\right) \in
A_{k}(Z) ,
\end{equation} where $$\gamma(S_{\alpha})=\mu_{S_{\alpha}}-\displaystyle\sum_{S_{\alpha'}\neq
S_{\alpha},\;\overline{S}_{\alpha' }\supset S_{\alpha}
}\gamma(S_{\alpha'})\,,$$ with
$\mu_{S_{\alpha}}=(-1)^{n}\;(\chi(F_{Z,[p]})-1)$ for some $[p] \in
S_{\alpha}$. In fact, equation (\ref{PP}) can be easily obtained
from equation (\ref{P}) by concentrating in degree $k$ in the Chow
group $A_*(Z).$

\vspace{0.2cm} Using  Theorem \ref{3} for
$X=\overline{S}_{\alpha}$ we have that the $k$-th
Schwartz-MacPherson class, $c_{k}^{SM}(\overline{S}_{\alpha})$, is
given by {\small $$\sum_{S_{\beta}\subset
\overline{S}_{\alpha}}\eta(S_{\beta},1_{\overline{S}_{\alpha}})\sum_{i=k+j}^{d_{\beta}}
(-1)^{d_{\beta}-i}
\left( \begin{array}{c}i+1\\
k+j+1\\
\end{array} \right)(i_{\overline{S}_{\beta},\overline{S}_{\alpha}})_{*}\left(c_{1}(\varphi_{\beta}^{*}{\cal L})^{i-k-j} \cap
[\mathbb{P}_{i}(\overline{S}_{\beta})]\right),$$} where
$\varphi_{\beta}=f\circ g\circ \psi_{\alpha} \circ
i_{\overline{S}_{\beta},\overline{S}_{\alpha}}$. Applying this in
equation (\ref{PP}), we have that

 {\small $$\begin{array}{lll}{\cal
M}_{k}(Z)&=&\displaystyle\sum_{\alpha}\gamma(S_{\alpha})
\sum_{j=0}^{d_{\alpha}-k}(-1)^{j}c_{1}(g^{*}f^{*}{\cal L})^{j}
\cap (\psi_{\alpha})_{*}\left(\displaystyle\sum_{S_{\beta}\subset
\overline{S}_{\alpha}}\eta(S_{\beta},1_{\overline{S}_{\alpha}})\right.\\
\vspace{0.3cm} ~&~&\left.\displaystyle\sum_{i=k+j}^{d_{\beta}}
(-1)^{d_{\beta}-i}
\left( \begin{array}{c}i+1\\
k+j+1\\
\end{array} \right)(i_{\overline{S}_{\beta},\overline{S}_{\alpha}})_{*}\left(c_{1}(\varphi_{\beta}^{*}{\cal L})^{i-k-j} \cap
[\mathbb{P}_{i}(\overline{S}_{\beta})]\right)\right).\end{array}$$}

Using:
\begin{itemize}
\item  That for all $\alpha, \beta$ we have $\eta(S_{\beta},1_{\overline{S}_{\alpha}})=0$ if $S_{\beta}\not\subset \overline{S}_{\alpha};$

That for all $\alpha, \beta$ we have $\eta(S_{\beta},1_{\overline{S}_{\alpha}})=0$ if $S_{\beta}\not\subset \overline{S}_{\alpha};$

\item that by a property of the  Euler characteristic we have:
$$\displaystyle\sum_{\beta}\gamma(S_{\beta})\eta(S_{\alpha},1_{\overline{S}_{\beta}})
=\eta\left(S_{\alpha},\displaystyle\sum_{\beta}\gamma(S_{\beta})1_{\overline{S}_{\beta}}\right)\,;$$

 \item and
 that   by  \cite[Lemma
4.1]{Par-Pr}) we have
$\displaystyle\sum_{\beta}\gamma(S_{\beta})1_{\overline{S}_{\beta}}=(-1)^{n}\;w$,
where $w(x)= \chi(F_{x})-1$;
\end{itemize}

\noindent we get: {\small $$\begin{array}{lll}\vspace{0.3cm}{\cal
M}_{k}(Z)&=&(-1)^{n} \sum_{j\geq 0}(-1)^{j}c_{1}(g^{*}f^{*}{\cal
L})^{j} \cap
\sum_{\beta}(-1)^{d_{\beta}}\eta(S_{\beta},w)\\
\vspace{0.3cm} ~&~&\displaystyle\sum_{i\geq
k+j}(-1)^{i} \left( \begin{array}{c}i+1\\
k+j+1\\
\end{array} \right)(\psi_{\alpha})_{*}(i_{\overline{S}_{\beta},\overline{S}_{\alpha}})_{*}\left(c_{1}(\varphi_{\beta}^{*}{\cal L})^{i-k-j} \cap
[\mathbb{P}_{i}(\overline{S}_{\beta})]\right).\end{array}$$}
\noindent That is, {\small $$\begin{array}{lll}\vspace{0.3cm}{\cal
M}_{k}(Z)&=&(-1)^{n}\sum_{j\geq 0}(-1)^{j}c_{1}(g^{*}f^{*}{\cal
L})^{j} \cap
\sum_{\beta}(-1)^{d_{\beta}}\eta(S_{\beta},w)\sum_{i\geq
k+j}(-1)^{i}\\
\vspace{0.3cm}
~&~& \left( \begin{array}{c}i+1\\
k+j+1\\
\end{array} \right)(\psi_{\beta})_{*}\left(c_{1}((\psi_{\beta})^{*}g^{*}f^{*}{\cal L})^{i-k-j} \cap
[\mathbb{P}_{i}(\overline{S}_{\beta})]\right).\end{array}$$}
\noindent Since
$$(\psi_{\beta})_{*}\left(c_{1}((\psi_{\beta})^{*}g^{*}f^{*}{\cal
L})^{i-k-j} \cap [\mathbb{P}_{i}(\overline{S}_{\beta})]\right)=
c_{1}(g^{*}f^{*}{\cal L})^{i-k-j} \cap
(\psi_{\beta})_{*}[\mathbb{P}_{i}(\overline{S}_{\beta})]\,,$$ we
have that
$$\begin{array}{lll}\vspace{0.3cm}{\cal M}_{k}(Z)&=&(-1)^{n}\displaystyle\sum_{j\geq 0}\displaystyle\sum_{i\geq
k+j}(-1)^{i+j}
\left( \begin{array}{c}i+1\\
k+j+1\\
\end{array} \right)
c_{1}(g^{*}f^{*}{\cal L})^{i-k} \; \cap\\
\vspace{0.3cm} ~&~&\cap \;
\displaystyle\sum_{\beta}(-1)^{d_{\beta}}\eta(S_{\beta},w)
(\psi_{\beta})_{*}[\mathbb{P}_{i}(\overline{S}_{\beta})]\\
\vspace{0.2cm} ~&=&(-1)^{n}\displaystyle\sum_{j\geq
0}\displaystyle\sum_{i\geq k+j}(-1)^{i+j}
\left( \begin{array}{c}i+1\\
k+j+1\\
\end{array} \right)
c_{1}(g^{*}f^{*}{\cal L})^{i-k} \cap MP^{\mathbb{P}}_{i}(w)
.\end{array}$$
 By equation (\ref{indice4}), since $L=f^{*}{\cal L}$, we  get:
$${\cal M}_{k}(Z)=(-1)^{n}\sum_{j\geq 0}\sum_{i\geq k+j}(-1)^{i+j}
\left( \begin{array}{c}i+1\\
k+j+1\\
\end{array} \right)
c_{1}(g^{*}L)^{i-k} \cap \left( (-1)^{n}
\Lambda^{\mathbb{P}}_{i}(w) \right).$$ Hence by equation (\ref{5})
we have: {\small $${\cal M}_{k}(Z)=(-1)^{n}\sum_{j\geq
0}\sum_{i\geq k+j}(-1)^{i+j}
\left( \begin{array}{c}i+1\\
k+j+1\\
\end{array} \right)
c_{1}(g^{*}L)^{i-k} \cap \left( (-1)^{n}
(i_{sing})_{*}\left(\Lambda_{i}(Z)\right) \right).$$}

\noindent Therefore, $${\cal M}_{k}(Z)=\sum_{j\geq 0}\sum_{i\geq
k+j}(-1)^{i+j}
\left( \begin{array}{c}i+1\\
k+j+1\\
\end{array} \right)
c_{1}(L|_{Z})^{i-k} \cap (i_{sing})_{*}\Lambda_{i}(Z),$$ as
stated.
\end{proof}

Next we have the following consequence of Lemma \ref{l: main}
which proves the first statement in Theorem \ref{principal}.

%We have the following equality in the Chow group of $Z$ (and also in that of $Z_{Sing}.$
\begin{lemma}\label{l: Milnor classes} The following equality
holds in the Chow group of $Z$:
$${\cal M}_{k}(Z)=\displaystyle\sum_{l=0}^{d-k}(-1)^{k+l}
\left( \begin{array}{c}l+k\\
k\\
\end{array}\right)
c_{1}(L|_{Z})^{l} \cap \Lambda_{l+k}(Z) ,$$ where
$d=\dim(Z_{sing}).$
\end{lemma}

\begin{proof}
We use the following identity:
\begin{equation}\label{combinatorica}
\left( \begin{array}{c}m\\
n
\end{array}\right)=\left( \begin{array}{c}m-1\\
n-1
\end{array}\right)+\left( \begin{array}{c}m-1\\
n
\end{array}\right) \,.
\end{equation}
We know:
$${\cal M}_{k}(Z)=\sum_{j\geq
0}\sum_{i\geq k+j}(-1)^{i+j}
\left( \begin{array}{c}i+1\\
k+j+1\\
\end{array} \right)
c_{1}(L|_{Z})^{i-k} \cap \Lambda_{i}(Z) .$$ Making $l=i-k$ we get:
$${\cal M}_{k}(Z)=\sum_{j=0}^{d-k}\sum_{l=j}^{d-k}(-1)^{k+j+l}
\left( \begin{array}{c}l+k+1\\
j+k+1\\
\end{array} \right)
c_{1}(L|_{Z})^{l} \cap \Lambda_{l+k}(Z) .$$ Re-arranging this
formula we obtain:
$${\cal M}_{k}(Z)=\sum_{l=0}^{d-k}\left [\sum_{j=0}^{l}(-1)^{k+j+l}
\left( \begin{array}{c}l+k+1\\
j+k+1\\
\end{array} \right)\right ]
c_{1}(L|_{Z})^{l} \cap \Lambda_{l+k}(Z) .$$ Hence we only need to
show:
$$\sum_{j=0}^{l}(-1)^{j}
\left( \begin{array}{c}l+k+1\\
j+k+1
\end{array} \right)=\left( \begin{array}{c}l+k\\
k
\end{array} \right) \,.$$
For this we use equation (\ref{combinatorica}); we get:
$$\begin{array}{lll} \vspace{0.3cm} \displaystyle\sum_{j=0}^{l}(-1)^{j}
\left( \begin{array}{c}l+k+1\\
j+k+1
\end{array} \right) &=& \displaystyle\sum_{j=0}^{l}(-1)^{j}
\left [\left( \begin{array}{c}l+k\\
j+k
\end{array} \right)+\left( \begin{array}{c}l+k\\
j+k+1
\end{array} \right)\right ]\\
\vspace{0.3cm} & = & \displaystyle\sum_{j=0}^{l}(-1)^{j}
\left( \begin{array}{c}l+k\\
j+k
\end{array} \right)+\displaystyle\sum_{j=0}^{l-1}(-1)^{j}\left( \begin{array}{c}l+k\\
j+k+1
\end{array} \right)\\
\vspace{0.3cm} & = & \displaystyle\sum_{j=0}^{l}(-1)^{j}
\left( \begin{array}{c}l+k\\
j+k
\end{array} \right)-\displaystyle\sum_{j=1}^{l}(-1)^{j}\left( \begin{array}{c}l+k\\
j+k
\end{array} \right)\\
\vspace{0.3cm} & = & \left( \begin{array}{c}l+k\\
k
\end{array} \right) \,,
\end{array}$$
which completes the proof of Lemma \ref{l: Milnor classes}.
\end{proof}

In order to prove the second statement in Theorem \ref{principal}
we need to formulate some technical results.

\begin{remark} {\rm Notice that: }
$$\sum_{k=m}^{n}(-1)^{k}
\left( \begin{array}{c}n\\
k\\
\end{array}\right)\left( \begin{array}{c}k\\
m\\
\end{array}\right)=\sum_{k=m}^{n}(-1)^{k}
\left( \begin{array}{c}n\\
n-k+m\\
\end{array}\right)\left( \begin{array}{c}n-k+m\\
m\\
\end{array}\right).$$

\end{remark}

The following lemma is an immediate consequence of the associative
law:
\begin{lemma}\label{somatorios} Let $a_{n,l}$ be elements in an
abelian (additive) group. Then,
$$\sum_{n=q}^{d}\sum_{l=0}^{d-n}a_{n,l}=\sum_{s=q}^{d}\sum_{t=0}^{s-q}a_{s-t,t} \;.$$
\end{lemma}

We  now prove  the  lemma below, which completes the proof of
Theorem \ref{principal}.

\begin{lemma}\label{secondTheorem1} We have the following equality in the Chow group of $Z$:
$${\Lambda}_{k}(Z)=\sum_{l=0}^{d-k}(-1)^{l+k}
\left( \begin{array}{c}l+k\\
k\\
\end{array}\right)
c_{1}(L|_{Z})^{l} \cap {\cal M}_{l+k}(Z) ,$$ where
$d=\dim(Z_{sing}).$
\end{lemma}

\begin{proof}
We will prove that we have:
$${\Lambda}_{d-k}(Z)=\sum_{l=0}^{k}(-1)^{d-k+l}
\left( \begin{array}{c}d-k+l\\
d-k\\
\end{array}\right)
c_{1}(L|_{Z})^{l} \cap {\cal M}_{d-k+l}(Z),$$ which implies the
lemma. Let  $$A_{d,k}=\sum_{l=0}^{k}(-1)^{d-k+l}
\left( \begin{array}{c}d-k+l\\
d-k\\
\end{array}\right)
c_{1}(L|_{Z})^{l} \cap {\cal M}_{d-k+l}(Z).$$ Then, making
$l=k-j$, we obtain:
$$A_{d,k}=\sum_{j=0}^{k}(-1)^{d-j}
\left( \begin{array}{c}d-j\\
d-k\\
\end{array}\right) \,.
c_{1}(L|_{Z})^{k-j} \cap {\cal M}_{d-j}(Z) \,.
$$
Using Lemma \ref{l: Milnor classes} we obtain:
$$A_{d,k}=\sum_{j=0}^{k}(-1)^{d-j}
\left( \begin{array}{c}d-j\\
d-k\\
\end{array}\right)
c_{1}(L|_{Z})^{k-j} \;\cap \qquad \qquad \qquad \qquad$$
 $$\qquad  \qquad  \cap \left [\sum_{l=0}^{j}(-1)^{d-j+l}
\left( \begin{array}{c}d-j+l\\
d-j\\
\end{array}\right)
c_{1}(L|_{Z})^{l} \cap \Lambda_{d-j+l}(Z) \right ].$$ Therefore
$$A_{d,k}=\sum_{j=0}^{k}\sum_{l=0}^{j}(-1)^{l}\left( \begin{array}{c}d-j+l\\
d-j\\
\end{array}\right)
\left( \begin{array}{c}d-j\\
d-k\\
\end{array}\right)
c_{1}(L|_{Z})^{k+l-j} \cap \Lambda_{d-j+l}(Z).$$ We now set
$d-j=n$, so  we get:
$$A_{d,k}=\sum_{n=d-k}^{d}\sum_{l=0}^{d-n}(-1)^{l}\left( \begin{array}{c}n+l\\
n\\
\end{array}\right)
\left( \begin{array}{c}n\\
d-k\\
\end{array}\right)
c_{1}(L|_{Z})^{k+l-d+n} \cap \Lambda_{n+l}(Z).$$ Then Lemma
\ref{somatorios} implies:
$$A_{d,k}=\sum_{s=d-k}^{d}\sum_{t=0}^{s-d+k}(-1)^{t}\left( \begin{array}{c}s\\
s-t\\
\end{array}\right)
\left( \begin{array}{c}s-t\\
d-k\\
\end{array}\right)
c_{1}(L|_{Z})^{s-d+k} \cap \Lambda_{s}(Z).$$ That is,
$$A_{d,k}=\sum_{s=d-k}^{d}(-1)^s\left [\sum_{t=0}^{s-d+k}(-1)^{s-t}\left( \begin{array}{c}s\\
s-t\\
\end{array}\right)
\left( \begin{array}{c}s-t\\
d-k\\
\end{array}\right)\right ]
c_{1}(L|_{Z})^{s-d+k} \cap\; \Lambda_{s}(Z).$$ Making $t=u-d+k$ we
obtain that $A_{d,k}$ equals the sum: {\small
$$ \sum_{s=d-k}^{d}(-1)^s\left [\sum_{u=d-k}^{s}(-1)^{s-u+d-k}\left( \begin{array}{c}s\\
s-u+d-k\\
\end{array}\right)
\left( \begin{array}{c}s-u+d-k\\
d-k\\
\end{array}\right)\right ]
c_{1}(L)^{s-d+k} \cap \;\Lambda_{s}(Z).$$}
 Hence the previous remark
implies:
$$A_{d,k}=\sum_{s=d-k}^{d}(-1)^s\left [\sum_{v=d-k}^{s}(-1)^{v}\left( \begin{array}{c}s\\
v\\
\end{array}\right)
\left( \begin{array}{c}v\\
d-k\\
\end{array}\right)\right ]
c_{1}(L|_{Z})^{s-d+k} \cap \Lambda_{s}(Z) \,.$$ So, using
\cite[Lemma (2.6)]{Verma}, we get :
$$A_{d,k}=\sum_{s=d-k}^{d}(-1)^s(-1)^{d-k}{\delta}_{s,d-k}\;
c_{1}(L|_{Z})^{s-d+k} \cap \Lambda_{s}(Z).$$ where
${\delta}_{m,n}$ as before denotes the  Kronecker  delta.
Therefore
$$A_{d,k} =\Lambda_{d-k}(Z),$$ as stated.
\end{proof}

\begin{remark}\label{Brasselet} It follows from the above proof  that one has the following
relations between the Milnor classes and polar cycles:
$$\begin{array}{rcl}\vspace{0.3cm}{\cal M}_{k}(Z)&=&\displaystyle\sum_{l=0}^{d-k}(-1)^{n+k+l}
\left( \begin{array}{c}l+k\\
k\\
\end{array}\right)
c_{1}(L)^{l}\; \cap
\displaystyle\sum_{\alpha}(-1)^{d_{\alpha}}\eta(S_{\alpha},w)
[\mathbb{P}_{l+k}(\overline{S}_{\alpha})],\end{array}$$ answering
in this way a question raised by J.-P. Brasselet in \cite {BraC}.
\end{remark}

%%%%%%%%%%%%%%%%

In a more general context in which $X$ is embedded as a closed
analytic subspace of a smooth complex
space $M$, P. Aluffi in \cite{Aluffi1} introduces a cycle
$\alpha_{X}$ in the Chow group, called later the Aluffi class,
which is related to the Donaldson-Thomas type invariants described
by K. Behrend in \cite{Beh}.

In the case we envisage here  $X$ is the singular set $Z_{sing}$
of a hypersurface $Z$ in a $n+1$-dimensional manifold $M$. We have
that $\alpha_{Z_{sing}}=c^{SM}(\nu_{Z_{sing}})$, with
$\nu_{Z_{sing}}(p)=(-1)^{n+1}(1-\chi(F_{p}))$, for all $p \in
Z_{sing}$. Note that $\nu_{Z_{sing}}=(-1)^{n}w$, as defined in the
Corollary \ref{local1}.

\begin{corollary}\label{c: Aluffi}
The Aluffi class $\alpha_{Z_{sing}}$ of $Z_{sing}$ relates to the
L\^e cycles as follows:
$$(\alpha_{Z_{sing}})_{k}=\sum_{l\geq
0}\sum_{m= 0}^{d-l-k}(-1)^{m+l+k}
{\small \left( \!\!\begin{array}{c}m+l+k\\
l+k\\
\end{array}\!\! \right)}
c_{1}(L)^{m+l} \cap \Lambda_{m+l+k}(Z).$$ Hence the Aluffi classes
relate  to the polar varieties as follows:
$$(\alpha_{Z_{sing}})_{k}=\sum_{l\geq
0}\sum_{m= 0}^{d-l-k}(-1)^{n+m+l+k}
{\small \left( \!\!\begin{array}{c}m+l+k\\
l+k\\
\end{array}\!\! \right)}
c_{1}(L)^{m+l} \; \cap $$ $$ \qquad   \cap \left(
\displaystyle\sum_{\alpha}(-1)^{d_{\alpha}}\eta(S_{\alpha},w)
[\mathbb{P}_{m+l+k}(\overline{S}_{\alpha})]\right).$$

\end{corollary}

\begin{proof} Given ${\cal S}=\{S\}$ a Whitney
stratification of $Z$ such that $Z_{sing}$ is a union of strata,
using  \cite[Lemma 4.1]{Par-Pr} we have that
\begin{equation}\label{A}\alpha_{Z_{sing}}=\sum_{S \in {\cal
S}}\;\gamma_{S}\;(i_{\overline{S},Z})_{*}c^{SM}(\overline{S}),\end{equation}
where $\gamma_{S}=(\nu_{Z_{sing}})_{S} - \displaystyle\sum_{S'
\neq S,\;\overline{S'} \supset S} \gamma_{S{'}}$ and
$i_{\overline{S},Z}:\overline{S}\hookrightarrow Z$ is the
inclusion. Hence by the description of the Milnor classes given by
A. Parusinski and P. Pragacz in \cite{Par-Pr} (see also equation
(\ref{P})), we have that

\begin{equation}\label{AluffiClass}\alpha_{Z_{sing}}=c(L_{|_{Z}}) \cap {\cal M}(Z).\end{equation}
 Thus the result
follows by the Theorem \ref{principal}. The second expression
follows similarly using the description of the Milnor classes via
polar varieties as in  Remark \ref{Brasselet}.
\end{proof}

We note that  in \cite{Aluffi1} the author relates the Aluffi
classes to the Milnor classes, and  formula (\ref{AluffiClass})
already appears in that article.  Yet, the proof we give here is
shorter and more direct (see
 \cite[Theorem 1.2]{Aluffi1}).

\section{On the top dimensional Milnor class}

Let ${\cal S}=\{S_{\alpha}\}$ be a Whitney stratification of $M$
with connected strata, adapted to $Z$ and to $Z_{\rm sing}$. Set
$d=\dim(Z_{\rm sing})$
 and assume that all $d$-dimensional
strata of $\cal S$ which lie in $Z_{\rm sing}$ are contained in
the non-singular locus of $Z_{\rm sing}.$ Notice that this
condition is not automatic unless we demand that the
stratification be adapted also to the singular set of $Z_{\rm
sing}$. Let $S_{\alpha}$ be a $d$-dimensional stratum contained in
$Z_{\rm sing}$,  let  $x$ be a point in $S_{\alpha}$,  so it is a
regular point of $ Z_{\rm sing},$ and let $H_x$ be a local
submanifold of $M$, biholomorphic to a disc of complex codimension
$d$, transversal at $x$ to the stratum $S_{\alpha}$. Then $H_x
\cap Z$ is a hypersurface in $H_x$ with an
isolated singularity at $x$. %The following definition is standard
%in the literature, recalling that the singular set is regular at
%$x$ (see for instance \cite{BLSS2}, also \cite{Massey0}).

\begin{definition}
Let $S_{\alpha}, x$ and $H_x$ be as before. The {\it transversal
Milnor number} of $S_{\alpha}$ at $x$ is the usual Milnor number
at $x$ of the isolated hypersurface germ $H_x \cap Z$. We denote
this number by $\mu^{\perp}(S_{\alpha},x)$.
\end{definition}

 The local topological triviality of Whitney stratifications implies
 that the transversal Milnor number $\mu^{\perp}(S_{\alpha},x)$ is constant
 along $S_{\alpha}$,
 and of course independent of the choice of the transversal slice $H_x$.
 Hence we denote this number just by $\mu^{\perp}(S_{\alpha})$ and call it the
 {\it transversal Milnor number of} $S_{\alpha}$. Notice that
 the number $\mu^{\perp}(S_{\alpha})$ is
 defined only for  $d$-dimensional strata  contained
 in the regular part of $Z_{\rm sing},$ where $d=\dim(Z_{\rm sing}).$  One has:

 \begin{lemma} \label{Lt}
 Let  $x \in Z_{\rm sing}$ be a regular point of the
singular set which belongs to a $d$-dimensional stratum $S_x$ of
$Z_{\rm sing}$. Set $\mu(Z,x)= (-1)^{n}\;(\chi(F_{x})-1) \,,$
where $F_x$ is a local Milnor fibre. Then the transversal Milnor
number $\mu^{\perp}(S_x)$ equals, up to sign, the  local Milnor
number
 $\mu(Z,x)$ of $Z$ at $x$. More precisely:
$$\mu(Z,x)=(-1)^{d}\mu^{\perp}(S_{x}).$$

 \end{lemma}

 \begin{proof}
 By definition we have that
$$\mu^{\perp}(S_{x})=\mu(H_{x}\cap Z,x)=\mu((s_{j})_{|_{H_{x}}},x),$$
where  $s_j$ is the restriction $s|_{U_j}$ of $s$ to a   local
chart  $U_j$ around $x$ in $Z,$ and $ H_{x} $ is as above, a local
submanifold of $M$, biholomorphic to a disc of complex codimension
$d$, transversal at $x$ to the stratum $S_x$.

Since $ S_{x} $ and $Z_{\rm sing}$ have the same dimension, $
H_{x} $ is transversal at $x$ to $ S_x$. Using the fact that every
Whitney stratified set is topologically locally trivial, we have
that there exists a local homeomorphism $(Z,x)\rightarrow
(S_{x},x) \times (N,x)$, with $N=H_{x} \cap Z$. Moreover, as every
holomorphic function satisfies the Thom condition (see
\cite[Section 5, Corollary 1]{Hironaka}), we conclude that
$$\chi(F_{s_{j},x})=\chi(F_{(s_{j})_{|_{H_{x}}},x})=1+(-1)^{n-d}\mu((s_{j})_{|_{H_{x}}},x),$$
where the last equality follows by the Milnor's Fibration Theorem.
Thus
$$\mu(Z,x)=(-1)^{n}\;(\chi(F_{s_{j},x})-1)=(-1)^{d}\;\mu((s_{j})_{|_{H_{x}}},x)=(-1)^{d}\mu^{\perp}(S_{x}).$$
\end{proof}

\begin{remark} \label{topMilnor} Let $S$ be a $d$-dimensional stratum
of a Whitney stratification of $M$ adapted to $Z,$ contained in
the non-singular locus of $Z_{\rm sing}$. We define the top
dimensional L\^e number $\lambda^{d}_{S}$ of $S$ by
$\lambda^{d}_{S}:=\lambda^{d}_{s_i,x},$ where
$\lambda^{d}_{s_i,x}$ is the Massey L\^e number of $s_i$ at $x,$
 for any  $x \in S\cap U_i$ and with $U_i$
a local chart around $x$ in $Z$ and $s_i$ the restriction
$s|_{U_i}$. Then, the transversal Milnor number $\mu^{\perp}(S)$
is equal to the top dimesional L\^e number $\lambda^{d}_{S}$. In
fact, since we may assume that $S$ is reduced, irreducible and
that for a generic point $x \in S$ we have that the multiplicity
of $S$ along $x$ is one, the claim follows directly by the
description of the top dimensional L\^e number
$\lambda^d_{s_{i}}(x)$ given by D. Massey in \cite[Proposition
2.8]{Massey0} (see also \cite[p. 20-21]{Massey}).
\end{remark}

Our last result is the corollary stated in the introduction:

\begin{corollary}
Assume that $L$ is a very ample line bundle on $M$. Then,

\begin{equation}\label{topMilnorFormula} {\cal M}_{d}(Z) \,
=\displaystyle\sum_{\begin{array}{c} S_\alpha \subseteq Z _{\rm
sing}\\{\rm dim}\;
S_{\alpha=d}\end{array}}\!\!\!(-1)^{d}\mu^\perp(S_\alpha)\;
[\overline{S}_{\alpha}] \, =\displaystyle\sum_{\begin{array}{c}
S_\alpha \subseteq Z _{\rm sing}\\{\rm dim}\;
S_{\alpha=d}\end{array}}\!\!\! (-1)^{d}\lambda^{d}_{S_\alpha}\;
[\overline{S}_{\alpha}]=(-1)^d\Lambda_d(Z),\end{equation}
where $\mu^\perp(S_\alpha)$
is the transversal Milnor number of $S_\alpha$ and
$\lambda^{d}_{S_\alpha}$ is the $d$-th L\^e number of $S_\alpha$.
This equality is as classes in the Chow group of $Z$ and hence
also in the singular homology group of $Z$.
\end{corollary}

\begin{proof}
Applying  Theorem \ref{principal} in the case $k = \dim
Z_{sing}=d$, we obtain:
$${\cal M}_{d}(Z)=(-1)^{d} \Lambda_{s}^{d},$$
where $\Lambda_{s}^{d}$ can be described by:
$$\sum_{\alpha}m_{\alpha}
(\psi_{\alpha})_{*}[\mathbb{P}_{d}(\overline{S}_{\alpha})]=
\displaystyle\sum_{{\rm dim}\;
S_\alpha=d}\;m_{\alpha}(\psi_{\alpha})_{*}[\overline{S}_{\alpha}]\,,$$
where $m_{\alpha}:=(-1)^{d-d_{\alpha}-1}
\chi(\phi_{{s_{j}}_{|_{N}}}F^{\bullet}_{|_{N}})_{x}$ with
$d_{\alpha}=\dim S_{\alpha}$ and $N$ is a local submanifold of $M$
of complex codimension $d_{\alpha}$, transversal at $x$ to
$S_{\alpha}$.

For $x$ in the top dimensional stratum $S_{\alpha}$, by
\cite[Lemma 4.13]{MasseyDuke} we have:
$$(-1)^{n+1-d-1}\chi(\phi_{{s_{j}}_{|_{N}}}F^{\bullet}_{|_{N}})_{x}=
\mu({s_{j}}_{|_{N}})\chi(F^{\bullet}_{|_{N}})_{x}=(-1)^{n+1-d}\mu({s_{j}}_{|_{N}}).$$
Hence $m_{\alpha}=\mu({s_{j}}_{|_{N}})=\mu^\perp(S_\alpha),$ which
proves the first equality in equation (\ref{topMilnorFormula}).
The second equality follows directly by Remark \ref{topMilnor}.
\end{proof}

\vspace{0.5cm}

{\bf Roberto Callejas-Bedregal}

 Centro de Ci\^encias Exatas e da Natureza - Campus I ,

 Universidade Federal da Para\'{i}ba-UFPb,

 Cidade Universit\'aria s/n Castelo Branco, Jo\~ao Pessoa, PB - Brasil.

{email: roberto@mat.ufpb.br}\\

\vspace{0.1cm}{\bf Michelle F. Z. Morgado}

Instituto de Bioci\^encias Letras e  Ci\^encias Exatas,

Universidade Estadual Paulista-UNESP,

Rua Crist\'ov\~ao Colombo, 2265, Jd. Nazareth, S. J. do Rio Preto,
SP - Brasil

{email: mfzmorgado@hotmail.com}\\

\vspace{0.1cm}{\bf Jos\'e Seade}

Instituto de Matem\'aticas, Unidad Cuernavaca,

Universidad Nacional Aut\'onoma de M\'exico-UNAM,

Av. Universidad s/n, Lomas de Chamilpa, C.P. 62210,

Cuernavaca, Morelos, M\'exico.

{email: jseade@matcuer.unam.mx}

\vspace{1cm}
\section{ L\^e cycles and Milnor classes: Erratum}

 In \cite {CMS} we introduced  a concept of global L\^e cycles of  singular hypersurfaces in compact complex manifolds, and we gave a formula (\cite[Theorem 1]{CMS}) relating these with the Milnor classes of the corresponding hypersurface.
   There are two subtle errors in \cite{CMS}, that we correct below. With this, the correct statement of Theorem 1 in \cite{CMS} is the following.

\begin{Thm} \label{principal}
Let $M$ be a compact complex manifold and $L$ a  very ample line
bundle on $M$. Consider the complex
analytic space $Z$ of zeroes  of a reduced non-zero holomorphic
section $s$ of $L$ and denote by $Z_{sing}$ the singular set of
$Z$.  Then the local L\^e cycles {\rm (introduced by Massey \cite{Massey})} at the points in $Z$ can be patched together giving rise to well defined classes $\Lambda_{k}(Z)$  of $Z$, $k= 0,\dots, r =\dim(Z_{sing})$,  in the Chow group of $Z$, that we call the global L\^e classes of $Z$, and these are related to the Milnor classes
 ${\cal M}_{k}(Z)$  by the formulas:
$${\cal M}_{k}(Z)= (-1)^{k}\displaystyle\sum_{s=0}^{r-k} C_{k,s}(d)\; c_{1}(L|_Z)^{s} \cap \Lambda_{k+s}(Z)\,$$
and conversely:
$$\Lambda_{k}(Z)= (-1)^{k}\displaystyle\sum_{s=0}^{r-k} B_{k,s}(d)\; c_{1}(L|_Z)^{s} \cap {\cal M}_{k+s}(Z)\,,$$
where
$d=\sqrt[n]{\deg_{L}[Z]}$, with $\deg_{L}[Z]=\displaystyle\int c_1(L)^{n-1} \cap [Z]$, and
$C_{k,s}(d)$ and $B_{k,s}(d)$ are rational functions on $d$  given by:
$$C_{k,s}(d)=\displaystyle\sum_{l=0}^{s}(-1)^{l}\left( \begin{array}{c}l+k-1\\
l \\
\end{array}\right)
 \left(\displaystyle\frac{(d-1)^{s+1-l}+(-1)^{s-l}}{d^{s+1}}\right)$$
 and
 $$B_{k,s}(d)=\displaystyle\sum_{l=0}^{s}(-1)^{s-l}\left( \begin{array}{c}k+s\\
k+l \\
\end{array}\right)
\left(\begin{array}{c}k+1+l\\
k+1 \\
\end{array}\right)
 \left(\displaystyle\frac{d-1}{d}\right)^{l}.$$
\end{Thm}

     We notice too that
   the Corollary following \cite[Theorem 1]{CMS} remains unchanged with minors obvious adaptations in the proof. Also, \cite[Theorem 2]{CMS} (see  \cite[Theorem 4.7]{CMS}) remains almost the same, we only need to change the polar and the MacPherson classes for the corresponding abstract classes and its proof have to have the same kind of modifications.
   \medskip

 The concept of global L\^e cycles of $Z$ that we gave in \cite{CMS} was
motivated by the Gaffney-Gassler interpretation in \cite{GG} of
 the local L\^e cycles of holomorphic map-germs introduced by David Massey in \cite{Massey}.
 These were  defined  as the Segre classes of $Z_{sing}$ in $M$, and we wrongly proved a theorem relating these to the Milnor classes of $Z$. Our proof used a characterization of the  L\^e cycles as  Massey (or MacPherson) cycles,  and that was the interpretation of the L\^e cycles that we actually used to prove our theorem. The first error in \cite{CMS} was precisely the identification of these Segre classes with the Massey, or MacPherson, cycles (equation (14) in \cite[Proposition 4.6]{CMS}), which is essential.
  Thence we must define the  global L\^e cycles differently, as indicated below.

  The second error concerns the use of Piene's formula relating Mather and Polar classes (see \cite{Piene}). That formula is valid only for closed subvarieties of projective spaces and not for closed embeddings into projective spaces as we used. Again, we can go round this problem as indicated below.

We are grateful to Paolo Aluffi for several useful comments and pointing out to us that our main theorem in \cite{CMS} could not be truth as stated because it  lead to contradictions.

\medskip
We now indicate  the corrections one needs to make on \cite{CMS}. For simplicity consider first the projective case.
Let  $X$ be a $D$-dimensional subvariety of  $\mathbb{P}^N$.  Its $k$-th polar variety is  \cite[Definition 4.2]{CMS}:
\begin{equation}\label{polar_definition} \mathbb{P}_{k}(X)=\overline{\{ x \in X_{reg} \; | \; {\rm dim} \left(T_{x}X_{reg}\cap L_{k+2}\right) \geq D-k-1 \}},\end{equation}
where $L_{k+2}$ is a sufficiently general  plane of  codimension ${k+2}$ in $\mathbb{P}^{N}$  and $T_xX_{reg}$ is the projective tangent space of $X$ at a regular point $x$.
For any given constructible function $\beta$ on $X\subseteq \mathbb{P}^{N}$, with respect to a Whitney stratification
${\cal S}=\{S_{\alpha}\}$ of $X$, the MacPherson cycles are  \cite[Equation (12)]{CMS}:
$$MP^{\mathbb{P}}_{k}(\beta):=\displaystyle\sum_{\alpha}(-1)^{{\rm
dim}\;S_{\alpha}}\eta(S_{\alpha},\beta)[\mathbb{P}_{k}(\overline{S}_{\alpha})],$$
where $ \eta(S_{\alpha},\beta)$ is the usual Normal Morse index (as in \cite[Definition 2.1]{CMS}).
Similarly, for every $F^{\bullet}$ in ${\cal D}^{b}_{c}(X)$,
the derived category of bounded, constructible complexes
of sheaves of $\C$-vector spaces on $X$,
the Massey cycles of $F^{\bullet}$  are  \cite[Definition 4.4]{CMS}:
$$\Lambda^{\mathbb{P}}_{k}(F^{\bullet})=\sum_{\alpha} m_{\alpha}
(\psi_{\alpha})_{*}[\mathbb{P}_{k}(\overline{S}_{\alpha})], $$
where $\psi_{\alpha}:\overline{S}_{\alpha}\hookrightarrow X$ is
the inclusion and
$m_{\alpha}=(-1)^{d-d_{\alpha}-1}\chi(\phi_{g_{|_{N}}}F_{|_{N}}^{\bullet})_{p}$\,;
 $g$ is a non-degenerate covector at $p \in S_\alpha$ with
respect to the fixed stratification ${\cal S}=\{S_\alpha\}$ and $N$ is a germ of a closed complex submanifold
which is transversal to $S_{\alpha}$ with $N \cap
S_{\alpha}=\{p\}$. When $\beta(p) =\chi(F^{\bullet})_{p}$
for $F^{\bullet}\in {\cal D}^{b}_{{\cal S}}(X)$ we denote
$\Lambda^{\mathbb{P}}_{k}(F^{\bullet})$ also by
$\Lambda^{\mathbb{P}}_{k}(\beta).$
By \cite[Proposition 4.6 (13)]{CMS} we have $$\;\Lambda^{\mathbb{P}}_{k}(F^{\bullet})=(-1)^{\dim X}MP^{\mathbb{P}}_{k}(F^{\bullet}).$$

The following result expresses each Schwartz-MacPherson class through MacPherson classes. The statement and the proof are as  in \cite[Theorem 4.7]{CMS}.

\begin{theorem} [{\bf projective case}] \label{Schwartz-MacPherson-projective}Let $X\subset \mathbb{P}^{N}$ be a subvariety endowed with a Whitney stratification with connected strata $S_{\alpha}$.
 If $\beta:X\rightarrow \Z$ is a constructible function with respect to this stratification and ${\cal L}:={\cal O}_{\mathbb{P}^{N}}(1)$, then
$$c_{k}^{SM}(\beta)=\sum_{\alpha}\eta(S_{\alpha},\beta)\sum_{i=k}^{d_{\alpha}}(-1)^{d_{\alpha}-i}
\left( \begin{array}{c}i+1\\
k+1\\
\end{array} \right)c_{1}({\cal L}|_X)^{i-k} \cap
[\mathbb{P}_{i}(\overline{S}_{\alpha})]$$
$$=\sum_{i\geq k}(-1)^{i}
\left( \begin{array}{c}i+1\\
k+1\\
\end{array} \right)c_{1}({\cal L}|_X)^{i-k} \cap
MP^{\mathbb{P}}_{i}(\beta). \qquad \;$$
\end{theorem}

\medskip

Now let $Z$ be a hypersurface in  $\mathbb{P}^{N}$ defined by a homogeneous polynomial $H$ of degree $d$, that is, $Z$
is the set of zeroes of a holomorphic section of ${\cal O}_{\mathbb{P}^{N}}(d)$. Denote by ${\cal M}_{k}(Z)$ the $k$-th Milnor class of $Z$ as defined in \cite[Definition 5.1]{CMS}.

\medskip

We may define the projective L\^e cycles of $Z$, $\Lambda_{k}(Z)$,  via \cite[Equation (13)]{CMS}:
\begin{equation}\label{Lecycle_definition}\Lambda_{k}(Z)= \Lambda^{\mathbb{P}}_{k}(\omega)=(-1)^{n-1}MP^{\mathbb{P}}_{k}(\omega),\end{equation} where $\omega(x)= \chi(F_{h,x})-1$ with $h$ being the restriction of $H$ to a neighborhood of $x$. As observed in \cite[Corollary 2.4]{CMS} and \cite[Remark 4.5]{CMS}, we have that $\Lambda_{k}(Z)$ restricted to any chart in an affine open covering of $\mathbb{P}^{N}$  coincides with the local L\^e cycle defined by Massey in \cite{Massey}.

With this definition of projective L\^e cycles,  Lemma 6.1 in \cite{CMS} becomes the lemma below (with essentially the same proof):

 \begin{lemma}[{\bf Main Lemma in the projective case}] \label{l: main} If $Z \subseteq \mathbb{P}^{N}$ then
$${\cal M}_{k}(Z)= (-1)^{n-1} \sum_{j\geq
0}\sum_{i\geq k+j}(-1)^{i+j}
\left( \begin{array}{c}i+1\\
k+j+1\\
\end{array} \right)
c_{1}(L|_Z)^{j} c_{1}({\cal L}|_Z)^{i-k-j} \cap MP^{\mathbb{P}}_{i}(\omega) .$$
and $${\cal M}_{k}(Z)=  \sum_{j\geq
0}\sum_{i\geq k+j}(-1)^{i+j}
\left( \begin{array}{c}i+1\\
k+j+1\\
\end{array} \right)
c_{1}(L|_Z)^{j} c_{1}({\cal L}|_Z)^{i-k-j} \cap \Lambda_{i}(Z) .$$

\end{lemma}

This leads to the first equation of Theorem \ref{principal}  for projective hypersurfaces.

\begin{example} {\rm Let $Z$ be the hypersurface of $\mathbb{P}^4$ defined by the homogeneous polynomial $H(x_0,\dots, x_4)=x_0 x_1.$ Let $U_i$ be the open chart of $\mathbb{P}^4$ given by $\{x_i\neq 0\}.$ Then $Z\cap U_i$ is defined by $h_i=x_0 x_1$ if $i=2,3,4$ and by $h_0=x_1; h_1=x_0.$  Let us compute the local L\^e cycles of $h_i$ for each $i=0,\dots, 4.$  For example, if $i=4$, then $\Sigma (h_4)=V(x_0,x_1),$ which is smooth. Then we have only local L\^e cycles in dimension $2$ given by $\Lambda_2(h_4)=V(x_0,x_1)$ and the other are all zero. The same argument works for $i=2,3.$ For $i=0,1$ we have that $\Sigma (h_i)=\emptyset,$ hence there are no local L\^e cycles for this case. Therefore, we have that $\Lambda_2(Z)=[\mathbb{P}^2]$ and the other are all zero.

If we compute the Milnor classes of $Z$ using the formula given in Lemma~\ref{l: main}, we will get that
${\cal M}_{k}(Z)=[\mathbb{P}^k]$ for $k=0,1,2.$

}
\end{example}

We now consider the general case.
Let $M$ be a compact complex manifold of dimension $n$, $L$ a  very ample holomorphic  line bundle  over $M$
and  $Z$  the complex analytic space of zeroes  of a reduced non-zero holomorphic
section $s$ of $L$. Denote by $Z_{sing}$ the singular set of
$Z$. One has a
closed embedding $\phi: M\hookrightarrow \mathbb{P}^{N}$ defined by the very ample line bundle $L$.
Using the notation ${\cal L}_N:={\cal O}_{\mathbb{P}^{N}}(1)$, for any $Y\subset M$ we have that $(\phi|_{Y})^{*}({\cal L}_N|_{\phi(Y)})=L|_{Y}$.

\begin{definition} Given  $Y \subseteq M$ of pure dimension $D$, its {\it abstract polar variety} is: $$P^{\mbox{abs}}_k(Y) :=(\phi|_{Y})^{-1}) \mathbb{P}_{k}(\phi(Y)), \;k=0, \dots, D,$$
where $\mathbb{P}_{k}(\phi(Y))$ is the projective polar variety  defined by equation (\ref{polar_definition}).
We have that $[P^{\mbox{abs}}_k(Y)] =((\phi|_{Y})^{-1})_{*} [\mathbb{P}_{k}(\phi(Y))]$ as cycles.
\end{definition}

\begin{definition}
For any given constructible function $\beta$ on $Y\subseteq M$, with respect to a Whitney stratification
$\{S_{\alpha}\}$ of $Y$,
define the {\it abstract MacPherson cycle} by:
$$MP^{\mbox{abs}}_{k}(\beta):=\displaystyle\sum_{\alpha}(-1)^{{\rm
dim}\;S_{\alpha}}\eta(S_{\alpha},\beta)[P^{\mbox{abs}}_{k}(\overline{S}_{\alpha})].$$

\end{definition}

We remark that although the abstract MacPherson cycles are defined similarly to the MacPherson cycles in the projective case, they can differ even for projective hypersurfaces.

Since $\phi: Y\rightarrow \phi (Y)$ is an isomorphism we have that
$$c_{k}^{Ma} (Y) = ((\phi|_{Y})^{-1})_{*} c_{k}^{Ma} (\phi(Y)).$$

The next  theorem is analogous to  \cite[Theorem 2]{CMS}.
%(and hence also to Theorem \ref{Schwartz-MacPherson-projective}), and so is its proof.
This is in fact the correct statement of \cite[Theorem 4.7]{CMS} in the general setting. We leave the details to the reader.

\begin{theorem} [{\bf general case}] \label{3}Let $Y\subset M$ be an n-dimensional subvariety endowed with a Whitney stratification with connected strata $S_{\alpha}$.
If $\beta:Y\rightarrow \Z$ is a constructible function with
respect to this stratification, then
$$c_{k}^{SM}(\beta)=\sum_{\alpha}\eta(S_{\alpha},\beta)\sum_{i=k}^{d_{\alpha}}(-1)^{d_{\alpha}-i}
\left( \begin{array}{c}i+1\\
k+1\\
\end{array} \right)c_{1}(L|_Z)^{i-k} \cap
[P^{\mbox{abs}}_{i}(\overline{S}_{\alpha})]$$
$$=\sum_{i\geq k}(-1)^{i}
\left( \begin{array}{c}i+1\\
k+1\\
\end{array} \right)c_{1}(L|_Z)^{i-k} \cap
MP^{\mbox{abs}}_{i}(\beta).  \qquad \;$$
\end{theorem}

As in \cite{CMS}, this yields to
the following lemma. The first  equality is  analogous to \cite[Lemma 6.1]{CMS}, and hence to Lemma \ref{l: main}. The second equality corresponds to \cite[Lemma 6.2]{CMS}.

\begin{lemma}[{\bf Main Lemma for the general case}]\label{l: Milnor classes}
$${\cal M}_{k}(Z)= (-1)^{n-1} \sum_{j\geq
0}\sum_{i\geq k+j}(-1)^{i+j}
\left( \begin{array}{c}i+1\\
k+j+1\\
\end{array} \right)
c_{1}(L|_Z)^{i-k} \cap MP^{\mbox{abs}}_{i}(\omega) ,$$

$$\qquad =(-1)^{n-1}\displaystyle\sum_{l=0}^{r-k}(-1)^{k+l}
\left( \begin{array}{c}l+k\\
k\\
\end{array}\right)
c_{1}(L|_{Z})^{l} \cap MP^{\mbox{abs}}_{l+k}(\omega). \qquad \qquad$$

Conversely, $$MP^{\mbox{abs}}_{k}(\omega)=(-1)^{n-1}\displaystyle\sum_{l=0}^{r-k}(-1)^{k+l}
\left( \begin{array}{c}l+k\\
k\\
\end{array}\right)
c_{1}(L|_{Z})^{l} \cap {\cal M}_{k+l}(Z).$$
\end{lemma}

\vspace{0,3cm}
\begin{remark} \label{Rem} If $Z$ is a hypersurface of  degree $d$ in $\mathbb{P}^{N}$, then the projective L\^e cycles $\Lambda_k(Z)$, as defined in equation (\ref{Lecycle_definition}), are related to the abstract MacPherson cycles by
$$\Lambda_k(Z)=(-1)^{n-1}\displaystyle\sum_{j=0}^{r-k}(-1)^{j}
\left( \begin{array}{c}k+1+j\\
k+1 \\
\end{array}\right)
(d-1)^{j} c_{1}({\cal L}_{n})^{j} \cap MP^{\mbox{abs}}_{k+j}(\omega),$$ as classes in the Chow group of $Z$. This motivates the following definition:
\end{remark}
%%%%%

\vspace{0,3cm}
\begin{definition}\label{globalLe} The {\it global L\^e classes}  of $Z$  are defined as:
$$\Lambda_k(Z)=(-1)^{n-1}\displaystyle\sum_{j=0}^{r-k}(-1)^{j}
\left( \begin{array}{c}k+1+j\\
k+1 \\
\end{array}\right)
\left(\frac{d-1}{d}\right)^{j} c_{1}(L|_Z)^{j} \cap MP^{\mbox{abs}}_{k+j}(\omega),$$ where $d=\sqrt[n]{\deg_{L}[Z]}$ with $\deg_{L}[Z]=\displaystyle\int c_1(L|_Z)^{n-1} \cap [Z]$.\end{definition}

Observe also  that in the projective setting we have  $\deg_{L}[Z]=d^{n}$ where $d$ is the usual degree of $Z.$
Notice that this definition describes the global L\^e classes in the homology group $H_{*}(Z,\R)$ with real coefficients.
It is clear, from Remark \ref{Rem},  that this definition extends the notion of the local L\^e cycles.

\medskip

By simple combinatoric manipulations we get:
$$ MP^{\mbox{abs}}_{k}(\omega)=(-1)^{n-1}\displaystyle\sum_{j=0}^{r-k}
\left( \begin{array}{c}k+1+j\\
k+1 \\
\end{array}\right)
\left(\frac{d-1}{d}\right)^{j} c_{1}(L|_Z)^{j} \cap \Lambda_{k+j}(Z).$$
Then by Lemma \ref{l: Milnor classes}  and  \cite[Lemma 6.4]{CMS}  we get:
$${\cal M}_{k}(Z)= \displaystyle\sum_{l=0}^{r-k}\;\;\;\displaystyle\sum_{j=0}^{r-l-k}(-1)^{l+k}\left( \begin{array}{c}l+k\\
k \\
\end{array}\right)
\left( \begin{array}{c}l+k+1+j\\
l+k+1 \\
\end{array}\right) \left(\frac{d-1}{d}\right)^{j} c_{1}(L|_Z)^{l+j} \cap \Lambda_{l+k+j}(Z).$$

$$ \qquad \quad= \,
 \displaystyle\sum_{s=0}^{r-k}(-1)^{k}\left(\frac{d-1}{d}\right)^{s}\left[\displaystyle\sum_{t=0}^{s}(-1)^{t}\left( \begin{array}{c}t+k\\
k \\
\end{array}\right)
\left( \begin{array}{c}s+k+1\\
t+k+1 \\
\end{array}\right) \left(\frac{d}{d-1}\right)^{t}\right] c_{1}(L|_Z)^{s} \cap \Lambda_{k+s}(Z).\\
$$
Notice that by simple combinatoric manipulations we get:
$$\displaystyle\sum_{t=0}^{s}(-1)^{t}\left( \begin{array}{c}t+k\\
k \\
\end{array}\right)
\left( \begin{array}{c}s+k+1\\
t+k+1 \\
\end{array}\right) \left(\frac{d}{d-1}\right)^{t}=\displaystyle\sum_{l=0}^{s}(-1)^{l}\left( \begin{array}{c}l+k-1\\
l \\
\end{array}\right)
 \left(\frac{(d-1)^{s+1-l}+(-1)^{s-l}}{d(d-1)^{s}}\right) \,,$$
and so we arrive to
\begin{equation} \label{Milnorgeneral}{\cal M}_{k}(Z)= (-1)^{k}\displaystyle\sum_{s=0}^{r-k} C_{k,s}(d)\; c_{1}(L)^{s} \cap \Lambda_{k+s}(Z).\end{equation}

Conversely, by Definition \ref{globalLe}, Lemma \ref{l: Milnor classes} and \cite[Lemma 6.4]{CMS}, we have that
$$\Lambda_{k}(Z)= (-1)^{k}\displaystyle\sum_{s=0}^{r-k} B_{k,s}(d)\; c_{1}(L|_Z)^{s} \cap {\cal M}_{k+s}(Z)\,,$$
where $B_{k,s}(d)=\displaystyle\sum_{l=0}^{s}(-1)^{s-l}\left( \begin{array}{c}k+s\\
k+l \\
\end{array}\right)
\left(\begin{array}{c}k+1+l\\
k+1 \\
\end{array}\right)
 \left(\displaystyle\frac{d-1}{d}\right)^{l}$ is a rational function in $d$.

\medskip
Recall that the Aluffi class $\alpha_{Z_{sing}}$ of the singular set $Z_{sing}$ of $Z$ can be described as in \cite[equation (21)]{CMS} by
$\alpha_{Z_{sing}}=c(L_{|_{Z}}) \cap {\cal M}(Z)$. Hence by equation (\ref{Milnorgeneral}) we have that
$$(\alpha_{Z_{sing}})_{k}=\sum_{l\geq
0}\sum_{s= 0}^{r-l-k}(-1)^{k+l}C_{k+l,s}(d)\; c_{1}(L)^{l+s} \cap \Lambda_{k+l+s}(Z).$$

%\end{remark}

 \enddocument

\end